\documentclass[11pt]{amsart}

\usepackage{tabularx}
\usepackage[numbers]{natbib}
\usepackage{amsthm,amsmath,mathrsfs}
\usepackage{amssymb,amsthm,amsmath}
\usepackage[dvips]{geometry}
\usepackage{amscd}
\usepackage{enumerate}
\usepackage{amsfonts}
\usepackage[latin1]{inputenc}
\usepackage[all] {xy}
\usepackage[normalem]{ulem}
\usepackage{mathtools}
\usepackage[dvips]{geometry}
\usepackage{enumitem}
\usepackage{graphicx}
\usepackage{subcaption}
\usepackage{soul,color}
\usepackage{tikz-cd}
\usepackage[titletoc,toc,title]{appendix}
\usepackage[colorlinks=true, linkcolor=red, citecolor=blue, urlcolor=blue]{hyperref}

\usepackage{amsthm}

\theoremstyle{plain}
\newtheorem{theorem}{Theorem}[section]

\newtheorem{lemma}[theorem]{Lemma}

\newtheorem*{main}{Main Theorem}

\newtheorem{claim}{Claim}
\theoremstyle{definition}
\newtheorem*{def*}{Definition}
\newtheorem{remark}[theorem]{Remark}

\newtheorem{definition}[theorem]{Definition}

\usepackage[pagewise]{lineno}

\newcommand{\R}{\mathbb{R}}
\newcommand{\C}{\mathscr{C}}

\author{Elias  Rego}
\address{Department of Applied Mathematics, AGH University of Science and Technology, Krakow, Poland.}
\email{rego@agh.edu.pl}

\author{Kendry J. Vivas}
\address{Departamento de Matem\'aticas, Universidad Cat\'olica del Norte, Antofagasta, Chile.}
\email{kendry.vivas01@ucn.cl}



\title{A trichotomy for generic sectional-hyperbolic chain-recurrent classes.}

\begin{document}

\begin{abstract}
The notion of sectional-hyperbolicity is a weakened form of hyperbolicity introduced for vector fields in order to understand the dynamical behavior of certain higher-dimensional systems such as the multidimensional Lorenz attractor. In this paper we address the  questions proposed in [\emph{Math. Z.}, \textbf{298} (2021), 469-488] and we provide a partial answer by proving that a $C^1$-generic non-trivial sectional-hyperbolic chain-recurrent class, not necessarily Lyapunov stable, satisfies a trichotomy: it is either a homoclinic loop, a union of saddle connections between singularities, or it is robustly a homoclinic class. 
\end{abstract}

\maketitle

\section{Introduction} 

Differentiable dynamical systems have become a fruitful and important research field in recent decades. The birth of this discipline is attributed to the seminal work of S. Smale \cite{S}, where the notion of hyperbolic set was introduced, presenting Smale's Horseshoe as a representative example. Since then, a very strong program towards the understanding of the global dynamics of dynamical systems has been initiated,  yielding  a rich theory where several important advances were achieved, including a better understanding of stability phenomena and of chaotic dynamics from both topological and statistical point of view. 
Later, motivated by the chaotic dynamics of the strange attractor in Lorenz's polynomial system
\begin{displaymath}
\left\lbrace \begin{array}{lll}
\dot{x}&=& \sigma(y-x) \\
\dot{y}&=& \rho x-y-xz \\
\dot{z}&=& xy-\beta z,
\end{array}\right.
\end{displaymath}
where $\sigma\approx 10$, $\beta\approx 8/3$ and $\rho\approx28$,  \cite{ABS} and \cite{G} introduced, independently, a geometric model now known as the geometric Lorenz attractor (GLA). One of its main features is the presence of a unique singularity that is accumulated by regular orbits, which causes the system to fail to be hyperbolic. Nevertheless, the GLA exhibits rich dynamical behavior reminiscent of hyperbolic sets, including a dense set of periodic orbits and robust transitivity. Moreover, it was shown in \cite{B} that the GLA is a homoclinic class. Motivated by this example, the notion of singular-hyperbolicity was introduced in \cite{Mo3}, where it is proven that singular-hyperbolic sets properly extend the classical notion of hyperbolic sets, incorporating the GLA as a prototypical case.

In order to gain a better understanding of the dynamics of higher-dimensional sets such as the multidimensional Lorenz attractor \cite{BPV}, C. Morales and R.J. Metzger introduced the concept of sectional-hyperbolic set in \cite{Me3}. Both notions, sectional-hyperbolicity and singular-hyperbolicity, agree for three-dimensional vector fields, but in the higher-dimensional setting sectional-hyperbolicity is stronger than singular-hyperbolicity (see \cite{Me3} and \cite{Sal} for more details).

Now, in topological dynamics, the chain-recurrent set constitutes an important subject of study because it encompasses all the interesting dynamics of the system. Moreover, by considering an equivalence relation, this  set is decomposed into pieces called chain-recurrence classes. The fundamental theorem of dynamical systems, due to Conley \cite{Co}, asserts the existence of a Lyapunov function which is constant along these classes. In this way, this result provides a general procedure for describing the global dynamics of a system. These sets are the central focus of this work and hence we now recall this concept. 

 Hereafter, we denote by $M$ a $n$-dimensional compact Riemannian manifold, endowed with a Riemannian metric $\Vert\cdot\Vert$. We will always assume $n\geq 4$. Denote by $d$ the metric on $M$ induced by its  Riemannian metric. Throughout this text, $\mathcal{X}^1(M)$ denotes the set of $C^1$-vector fields on $M$ endowed with the $C^1$-topology.  It is well known that any $X\in \mathcal{X}^1(M)$ induces a $C^1$-flow that will be denoted by the one-parameter family of maps $\lbrace X_t\rbrace_{t\in\mathbb{R}}$. The {\it{orbit}} of a point $x\in M$ is the  set $$\mathcal{O}(x)=\lbrace X_t(x) : t\in\mathbb{R}\rbrace.$$ For $a,b\in\mathbb{R}$, the \textit{orbit segment from $a$ to $b$} of a point $x$ is defined by $X_{[a,b]}(x)=\lbrace X_t(x) : t\in[a,b]\rbrace$. A point $x\in M$ is said to be a \textit{singularity} of  $X$ if $X(x)=0$.  We will denote the set of singularities of $X$ by $Sing(X)$. A point $x\in M\setminus Sing(X)$ is a  {\it{periodic point }} of $X$ if there is $t>0$ such that $X_t(x)=x$. Denote by $Per(X)$ the set of periodic points of $X$. The set of \textit{critical elements} of $X$ is given by $Crit(X)=Sing(X)\cup Per(X)$. An orbit that does not belong to $Crit(X)$ is called a \textit{regular orbit}. As usual, we say that a subset  $\Lambda$ of $M$ is {\it{invariant}} if $X_t(\Lambda)=\Lambda$  for any $t\in\mathbb{R}$. We say that a compact invariant set $\Lambda$ is {\it{Lyapunov stable}} if for every neighborhood $U$ of $\Lambda$ there exists a neighborhood $V\subset U$ of $\Lambda$ such that $X_t(V)\subset U$, for any $t>0$. We say that a compact and invariant set $\Lambda$ is transitive if it contains a point whose orbit is dense in $\Lambda$. 
 
Let us now define the chain-recurrent sets. There are several equivalent ways of defining the chain-recurrent set (see \cite[Theorem 2.7.18]{AN}), here we choose the form that is most suitable to our purposes. For $\varepsilon>0$ and  $T>0$, we say that a finite sequence $( x_i,t_i)_{i=0}^n$ is an \textit{$(\varepsilon, T)$-chain} if  $t_0+\cdots+t_n\geq T$, $t_i\geq 1$, and $d(X_{t_i}(x_i), x_{i+1}) <\varepsilon$ for any $i=0,..., n-1$. Besides, we say that $y$ is \textit{chain attainable from} $x$, and we denote it by $x\sim y$, if for any $\varepsilon,T>0$ there exists an $(\varepsilon,T)$-chain from $x$ to $y$, i.e. $x_0=x$ and $x_n=y$. When $x\sim x$, we say that $x$ is a \textit{chain-recurrent point}. The chain-recurrent set of $X$ is defined as $$CR(X)=\lbrace x\in M: x\sim x\rbrace.$$ 
It is well known that $CR(X)$ is a compact and invariant set. Moreover, $\sim$ is an equivalence relation on $CR(X)$. So, each equivalent class under this relation is called \textit{chain-recurrent class}. It is easy to see that each chain-recurrent class of $X$ is also a compact and invariant set. When $x\in CR(X)$, we denote by $C(x)$ its chain-recurrence class. We say that a chain-recurrent class is \textit{non-trivial} if it is not reduced to either a periodic orbit or a singularity.  We say that a chain-recurrent class is an \textit{aperiodic class} if it does not contain periodic orbits.

Several generic properties of these chain-recurrence classes were obtained. For instance, the well-known Kupka-Smale Theorem states that critical orbits of generic vector fields are hyperbolic. Later, in \cite{BC}, it was shown that for $C^1$-generic systems, chain-recurrence classes with periodic orbits coincide with the homoclinic class of some of such periodic orbits.   Moreover, when some form of hyperbolicity is present, we obtain more interesting dynamical properties. Indeed, according to \cite{PYY}, for $C^1$-generic vector fields, a non-trivial Lyapunov stable chain-recurrence class that is sectional-hyperbolic is necessarily a transitive attractor, and hence, a homoclinic class. In \cite{GYZ}, the same conclusion was obtained for chain-recurrence classes, not necessarily singular-hyperbolic, associated to $C^1$-generic vector fields away from homoclinic tangencies.  In \cite{CY}, S. Crovisier and D. Yang proved that transitivity is present in a robust way for sectional-hyperbolic Lyapunov stable chain-recurrence classes associated to vector fields $X$ in a certain $C^1$-generic set. In light of this result, the authors posed the following question:

\vspace{0.1in}
\textbf{Question:} Let $M$ be a compact Riemannian manifold with dimension at least $4$. Does there exist an open and dense set $\mathcal{U}\subset\mathcal{X}^1(M)$ such that for any $X\in \mathcal{U}$, any non-trivial sectional-hyperbolic chain-recurrent class is robustly a homoclinic class?
\vspace{0.1in}

In this work, we explore this question, and we give a partial answer to this question. 
In order to state our main results, let us begin by recalling the concept of sectional-hyperbolic set. We say that a compact invariant set $\Lambda$ has a {\it{dominated splitting}} if there is a continuous invariant  splitting $T_{\Lambda}M=E\oplus F$ (with respect to the tangent flow $\Phi_t= DX_t$, $t\in\mathbb{R}$) and  constants $K,\lambda>0$ satisfying the relation
\begin{displaymath}
\frac{\Vert \Phi_t(x)\vert_{E_x}\Vert}{m(\Phi_t(x)\vert_{F_x})}\leq Ke^{-\lambda t}, \quad \forall x\in\Lambda,\forall t>0,
\end{displaymath} 
where $m(A)$ denotes the co-norm of a linear transformation $A$. In this case, we say that $E$ is {\it{dominated}} by $F$. When the subbundle $E$ is uniformly contracting, i.e., $\Vert \Phi_t(x)\vert_{E_x}\Vert\leq Ke^{-\lambda t}$ for every $t>0$ and $x\in\Lambda$, we say that $\Lambda$ is \textit{partially hyperbolic}. 
In \cite{PYY} it was defined the notion of \textit{sectional-hyperbolic} set as follows:
\begin{definition}\label{def: sec-hyp}
    A compact invariant set $\Lambda\subset M$ is sectional-hyperbolic if it is partially hyperbolic and its central bundle  $F$ is  \textit{sectional-expanding}, i.e, there are $K,\lambda>0$ such that for every two-dimensional subspace $L_x$ of $F_x$ one has
\begin{displaymath}
\vert \text{det}\Phi_t(x)\vert_{L_x}\vert\geq Ke^{\lambda t},\quad\forall x\in\Lambda,\forall t> 0.
\end{displaymath}    
\end{definition}

\begin{remark}\label{rmk}
The notion of sectional-hyperbolicity was first introduced in \cite{Me3} to describe the dynamical behavior of higher-dimensional systems, such as the multidimensional Lorenz attractor (see \cite{BPV}). That definition differs from Definition \ref{def: sec-hyp} in that it assumes the hyperbolicity of singularities, a condition we do not require here. However, as we shall show in Lemma \ref{lemma: hypsing}, every singularity contained in a sectional-hyperbolic chain-recurrent class (under Definition \ref{def: sec-hyp}) is necessarily hyperbolic. Thus, this distinction poses no real restriction when dealing with chain-recurrence classes.
\end{remark}

Next, let us clarify the notion of  \emph{robustly homoclinic class}.
it is immediate to verify that \( \operatorname{Crit}(X) \subset CR(X) \), so the chain-recurrent class of any critical element is well defined. Moreover, by the previous remark and the hyperbolic lemma (see Lemma \ref{hyplemma}), every critical element of a sectional-hyperbolic chain-recurrent class is hyperbolic. If \( \gamma_X \) is a critical element of \( X \) and \( C(\gamma_X) \) is sectional-hyperbolic, then for every \( Y \in \mathcal{X}^1(M) \) sufficiently \( C^1 \)-close to \( X \), there exists a continuation \( \gamma_Y \) of \( \gamma_X \). The corresponding chain-recurrent class \( C(\gamma_Y) \) is called the \emph{continuation} of \( C(\gamma_X) \). We say that the chain-recurrent class \( C(\gamma_X) \) is \emph{robustly a homoclinic class} if there exists an open \( C^1 \)-neighborhood \( \mathcal{U} \) of \( X \) such that for every \( Y \in \mathcal{U} \), the continuation \( C(\gamma_Y) \) is a homoclinic class. The  main result of this paper is the following: 

\begin{main}\label{mainA}
There exists a $C^1$-generic set $\mathcal{R}\subset \mathcal{X}^1(M)$ such that any non-trivial sectional-hyperbolic chain-recurrent class $\Lambda$ satisfies the following trichotomy:
\begin{enumerate}
    \item $\Lambda$ is a homoclinic loop. 
    \item $\Lambda$ consists of saddle connections between singularities of $\Lambda$. 
    \item $\Lambda$ is robustly a homoclinic class. 
\end{enumerate} 
\end{main}
  Notice that in the main result above we are not assuming the class is Lyapunov stable. Indeed, if the class is Lyapunov stable, then \cite{CY} implies that the trichotomy is reduced to case (3) of the Main Theorem. Removing this assumption imposes substantial obstacles to the analysis of generic chain-recurrent classes (see \cite{PYY,GYZ, PYY2}). In this direction, our result allows one to recover robust homoclinic classes, provided the original classes are neither homoclinic loops nor a saddle connections. This shows that sectional-hyperbolicity, rather than Lyapunov stability, is the fundamental mechanism for the emergence of homoclinic structure in the generic setting.

The remainder of this text is organized as follows: In Section 2, we present a collection of preliminary results concerning sectional-hyperbolic chain-recurrence classes. In Section 3, we prove that any chain-recurrent class associated with $X\in\mathcal{R}$ which is not a homoclinic loop or a saddle connection is robustly periodic, which constitutes a key step toward proving the Main Theorem. Finally, Section 4 is devoted to proving our main result.

\section{Preliminary results on Sectional-Hyperbolic chain-classes}

In this section, we collect some tools that will be used in our analysis. Let us begin by recalling that  an important fact about chain-recurrence classes is that they are chain transitive: If $C$ is a chain-recurrent class, then for any pair of points $x,y\in C$  the point $y$ is chain attainable from $x$ through points in $C$, i.e., for every $\varepsilon, T>0$, there is an $(\varepsilon,T)$-chain $(x_i,t_i)_{i=0}^n$ such that $x_0 = x, x_n=y$ and   $x_i\in C$ for any $i=0,\ldots,
 n$ (see for instance \cite[Proposition 2.1]{FS}).

Next, we recall the concept of hyperbolic set. A compact invariant set $\Lambda$ is said to be a {\it hyperbolic set} if there is a $\Phi_t$-invariant splitting $T_\Lambda M=E^s\oplus \langle X\rangle \oplus E^u$, such that $E^s$ is of contracting type, $E^u$ is of expanding type and $\langle X\rangle $ is the direction generated by $X$. A hyperbolic set is of {\it saddle type} if both $E^s$ and $E^u$ are non-trivial. A critical element of $X$ is hyperbolic if its orbit is a hyperbolic set. Notice that when $x$ is a singularity, the subspace $\langle X \rangle$ is trivial. It is well known that sectional-hyperbolic sets satisfy the so-called hyperbolic lemma, whose proof can be found in \cite{BM}:  
\begin{lemma}[Hyperbolic Lemma]\label{hyplemma}
Every compact invariant set without singularities contained in a sectional-hyperbolic set is hyperbolic of saddle type. 
\end{lemma}
The above lemma shows that this notion generalizes the concept of uniform hyperbolicity introduced by Smale in \cite{S}. Furthermore, it tells us that this notion is good enough to obtain interesting dynamical results. 

In what follows, we prove some basic lemmas that will be instrumental in the proof of our main theorems. First, we prove the assertion made in Remark \ref{rmk}.

\begin{lemma}\label{lemma: hypsing}
Let $\Lambda$ be a non trivial sectional-hyperbolic chain-recurrent class for a $C^1$ vector field $X$ containing singularities. Then, every $\sigma\in \Lambda$ is hyperbolic. Moreover, if $T_{\sigma}M=E^{s}_\sigma\oplus E^u_\sigma$ is the hyperbolic splitting of $\sigma$ and $T_{\sigma}M=E_\sigma\oplus F_\sigma$ is the hyperbolic splitting of $\sigma$ given by sectional-hyperbolicity, we have
\begin{displaymath}
dim\,E_\sigma+1=dim\,E^s_\sigma. 
\end{displaymath}
\end{lemma}
\begin{proof}
 Assume that there is a complex eigenvalue $\lambda$ of $DX(\sigma)$ with zero real part, and let $L_{\lambda}=\text{span}\lbrace v_1,v_2\rbrace$ be the subspace spanned by the eigenvectors associated to $\lambda$. Then, $L_{\lambda}$ is invariant for $\Phi_t(\sigma)$ and there are $a,b>0$ such that $a<\Vert \Phi_t(\sigma)v_i\Vert<b$, $i=1,2$, for every $t>0$, so that $L_{\lambda}\cap E_\sigma=\lbrace 0\rbrace$. Therefore, since $E_\sigma$ is dominated by $F_\sigma$, we conclude that $L_{\lambda}\subset F_\sigma$, which contradicts the sectional expanding property of $F$. This shows that real part of any eigenvalue of $DX(\sigma)$ must be non-zero.

Now let consider the hyperbolic splitting $T_{\sigma}M=E^s_\sigma\oplus E^u_\sigma$ of $\sigma$. Note that each one of this subspaces is non trivial, otherwise we have that $\sigma$ is either a source or a sink, so that $\Lambda$ would be trivial which is impossible. Besides, since vectors in $E(\sigma)$ contract uniformly, we see by the domination property  that $E\subseteq E^s_\sigma$ and $E^u_\sigma\subseteq F$.

We claim that $E\neq E^s_\sigma$. Indeed, since $\Lambda$ is invariant and the splitting of $T_{\Lambda}M$ is dominated, it follows that the flow direction $X(x)$ belongs to $F_x$ for any regular point $x\in\Lambda$. Indeed, without loss of generality, we can assume that the Riemannian metric left the subbundles $E$ and $F$ orthogonal. In this way, for $x \in \Lambda$ we write $X(x)=v_{E}+v_{F}$, with $v_{E} \in E_x$ and $v_{F} \in F_x$. Assume that $X(x) \notin F_x$ for some $x\in \Lambda$, then $v_E \neq 0$. Since $E$ and $E^c$ are $\Phi_t$-invariant and orthogonal we have that for every $t \in \R$ that
\begin{eqnarray*} \label{X(x)notinEC}
K&\geq&\Vert X(X_t(x)) \Vert^2\\
&=&\Vert \Phi_t(x) X(x)\Vert^2\\
&=&\Vert \Phi_t(x) v_{E}\Vert^2 + \Vert \Phi_t v_{F}\Vert^2\\
&\geq& \Vert \Phi_t(x) v_{E}\Vert^2\geq \frac{1}{K}e^{\lambda t}\Vert v_e \Vert \rightarrow \infty,\quad t \rightarrow -\infty. 
\end{eqnarray*}
which is a contradiction. 

On the other hand, since $\Lambda$ is partially-hyperbolic it is well known that any point $x\in\Lambda$ has associated its strong stable manifold $\mathcal{F}_{ss}(x)$ which is invariant and it is tangent to $E$ at $x$ (see \cite{HPPS} for instance). Thus, if $x\in\Lambda\cap\left( \mathcal{F}_{ss}(\sigma)\setminus\lbrace\sigma\rbrace\right)$, then $X(x)\in T_x\mathcal{F}_{ss}(\sigma)=E_x^s$, which is absurd. Therefore, if $E=E^s_\sigma$ we have that $W^s(\sigma)=\mathcal{F}_{ss}(\sigma)$, so that 
\begin{displaymath}
\Lambda\cap (W^s(\sigma)\setminus\lbrace\sigma\rbrace)=\Lambda\cap (\mathcal{F}_{ss}(\sigma)\setminus\lbrace\sigma\rbrace)=\emptyset,
\end{displaymath}
which is a contradiction because $\Lambda$ is a non trivial chain-recurrent class. So, by domination property we see that $E^s_\sigma\cap F_{\sigma}\neq \lbrace0\rbrace$. Hence, since $F$ is area expanding, it follows that $dim\,(E^s_\sigma\cap F_{\sigma})=1$, which proves the result.  
\end{proof}

\begin{remark}\label{rmk: LL}
A singularity $\sigma$ satisfying the conditions of above lemma is called \textit{Lorenz-like singularity}. Moreover, the conditions $dim E^s_\sigma\geq 2$ and $dim\,(E^s_\sigma\cap F_{\sigma})=1$ provide an alternative definition of Lorenz-like singularity: A singularity $\sigma$ for $X$ is Lorenz-like if there are at least two eigenvalues with negative real part, one of them is real, say $\lambda_-,$ and the real part of remaining eigenvalues are outside the interval $[\lambda_-,-\lambda_-]$.
\end{remark}


We end this subsection by showing that sectional-hyperbolicity of a compact invariant set can be extended to the maximal invariant set of a neighborhood of it.  

\begin{lemma}\label{SHneighborhood}
Let $\Lambda$ be a sectional-hyperbolic set for a $C^1$ vector field $X$. Then, there is a neighborhood $U$ of $\Lambda$ such that the maximal invariant set $\widetilde{\Lambda}$ in $U$ is also sectional-hyperbolic for $X$.  
\end{lemma}
\begin{proof}
The proof of this result is based in the proof of  Propositions 3.1 and 3.2 in \cite{AM} and Proposition 2.10 in \cite{AM2}. First, there is a neighborhood $U_0$ of $\Lambda$ and a continuous (not necessarily invariant) extension $T_{U_0}M=\widetilde{E}\oplus\widetilde{F}$ of the splitting $T_{\Lambda}M=E\oplus F$. Moreover, if $\Lambda'$ is the maximal invariant set for the flow of $X$ contained in $U_0$, there are constants $a,T>0$ such that for every $x\in\Lambda'$ the cone fields 
\begin{displaymath}
\C_x^s(a)=\lbrace v=v^{\widetilde{E}}+v^{\widetilde{F}}\in \widetilde{E}\oplus \widetilde{F} : \Vert v^{\widetilde{F}}\Vert\leq a\Vert v^{\widetilde{E}}\Vert\rbrace
\end{displaymath}
and 
\begin{displaymath}
\C_x^c(a)=\lbrace v=v^{\widetilde{E}}+v^{\widetilde{F}}\in \widetilde{E}\oplus \widetilde{F} : \Vert v^{\widetilde{E}}\Vert\leq a\Vert v^{\widetilde{F}}\Vert\rbrace
\end{displaymath}
satisfy the following conditions: 
\begin{enumerate}
    \item Invariance: 
\begin{displaymath}
\Phi_{-t}(\C_{\phi_t(x)}^s(a))\subset \C_x^s(a)\text{ and } \Phi_{t}(\C_{x}^c(a))\subset \C_{\phi_t(x)}^c(a),\quad \forall t\geq T.
\end{displaymath}
\item Domination: There are constants $c>0$ and $\lambda_0\in(0,1)$ such that 
\begin{displaymath}
    \frac{\Vert \Phi_t(x)u\Vert}{\Vert \Phi_t(x)v\Vert}\leq c\lambda_0^{t}\frac{\Vert u\Vert}{\Vert v\Vert},\quad \forall t>0,
\end{displaymath}
for every nonzero vectors $v\in \C_{x}^c(a)$ and $u\in \Phi_{-t}(\C_{\phi_t(x)}^s(a))$. 
\end{enumerate}

Now, for $x\in\widetilde{\Lambda}$ define the bundles
\begin{displaymath}
E^s=\bigcap_{t\geq 0}\Phi_{-t}(\C_{X_t(x)}^s(a))\text{ and }E^c=\bigcap_{t\geq 0}\Phi_{t}(\C_{X_{-t}(x)}^c(a)).
\end{displaymath}
By Proposition 3.2 in \cite{AM} one has that $E^s$ is a $\Phi_t$-invariant and uniformly contracting continuous bundle satisfying $E^s_x\subset \C_x^s(a)$ for every $x\in\Lambda'$ and  $E^s_x=E_x$ for every $x\in\Lambda$. Besides, since $\Phi_{t}(\C_{X_{-t}(x)}^c(a))\subset \C_x^c(a)$ for every $t\geq T$ we have that $E^c_x\subset \C_x^c(a)$. So, by construction of the bundles and domination property, we have that $T_{\Lambda}M=E^s\oplus E^c$ is a dominated splitting for $\Lambda'$.   

\textbf{Claim: }$E^c$ is a $\Phi_t$-invariant sectional expanding continuous bundle such that $E^c_x=F_x$ for every $x\in\Lambda$.

On one hand, let's see that $E_x^c$ is an element of the space $\mathcal{G}_x$ of $d_F$-dimensional subspaces of $T_xM$. First, since $\Lambda'$ is compact, the space $\mathcal{G}_x$ is compact, so that the sequence $\lbrace\Phi_t(\widetilde{F}_{X_{-t}(x)})\rbrace_{t>0}\subset \mathcal{G}_x$ has a subsequence $\Phi_{t_n}(\widetilde{F}_{X_{-t_n}(x)})$, $t_n\geq T$, whose limit $G_x$ belongs to $\mathcal{G}_x$. Second, since $\Phi_{t}(\C_{X_{-t}(x)}^c(a))$ is a nested family of subspaces of $T_xM$, by definition of $E^c$ we have that $E^c=\lim_{t\to\infty}\Phi_{t}(\C_{X_{-t}(x)}^c(a))$. 
Then, as $\widetilde{F}_x\subset \C_x^c(a)$, it follows that $G_x\subset E^c_x$. In fact, $E^c_x=G_x$ for every $x\in\Lambda'$. Indeed, assume that $E_x^c\neq G_x$ for some $x\in\Lambda'$, so that $E_x^c$ is a non-trivial cone. Let $u_n\in \widetilde{F}_{X_{-t_n}(x)}$, $n\geq 1$, such that $\Vert\Phi_{t_n}(X_{-t_n}(x))u_n\Vert=1$. So, there is a nonzero vector $v\in E^s_x$ such that $w_n=\Phi_{t_n}(X_{-t_n}(x))u_n+v\in E_x^c$ for every $n\geq 1$, so that $\Phi_{-t_n}(x)w_n=u_n+\Phi_{-t_n}(x)v\in \C_{X_{-t_n}(x)}^c(a)$ by definition of $E_x^c$. Hence,   
\begin{displaymath}
\Vert \Phi_{-t_n}(x)^{\widetilde{E}}v\Vert\leq a(1-a^2)^{-1}\Vert u_n\Vert.
\end{displaymath}
Then, by domination and the definition of $u_n$, 
\begin{eqnarray*}
c^{-1}(\lambda_0)^{-t_n}\Vert u_n\Vert\Vert v\Vert\leq\Vert \Phi_{-t_n}(x)v\Vert&\leq &(1+a)\Vert \Phi_{-t_n}(x)^{\widetilde{E}}v\Vert \\
&\leq &a(1+a)(1-a^2)^{-1}\Vert u_n\Vert,
\end{eqnarray*}
which implies that 
\begin{displaymath}
\Vert v\Vert\leq c(\lambda_0)^{t_n}a(1+a)(1-a^2)^{-1}\to 0,\quad n\to\infty.
\end{displaymath}
This contradicts the choice of $v$.  

On the other hand, since $F$ is invariant on $\Lambda$ it follows from definition of $E^c$ that $E_x^c=F_x$ for every $x\in\Lambda$. By continuity of $\Phi_t$ and the definition of $E^c$ we have that $E^c$ is continuous and invariant. 

Finally, we show the sectional expanding property of $E^c$. Let $T_0>0$ such that $\lambda'=\ln((1/2)Ke^{T_0})>0$, where $K$ is given by sectional-hyperbolicity of the bundle $F$. By continuity of $E^c$ and $\Phi_t$ for $x\in\Lambda$ there is a neighborhood $V_x\subset G(2,F)$, where $G(2,F)$ is the Grasmannian bundle of two-planes contained in $E^c$, of $(x,L_x)$ such that for any $(y,L_y)\in V_x$,
\begin{displaymath}
\vert\det\Phi_t(x)\vert_{L_y}\vert\geq \frac{1}{2}\vert\det\Phi_t(x)\vert_{L_x}\vert\geq Ke^{\lambda t},\quad\forall t\in[0,T_0],
\end{displaymath}
where $K=C/2>0$. So, by compactness of $\Lambda$ and $G(2,F)$, and by shrinking $U_0$ if it is necessary, we have  that if $y\in\Lambda'$ and $L_y\subset E_y^c$, there are $x\in \Lambda$ and a two-dimensional subspace $L_x\subset F_x$ such that      
\begin{equation}\label{eq1}
\vert\det\Phi_t(y)\vert_{L_y}\vert\geq Ke^{\lambda t},\quad \forall t\in[0,T_0]. 
\end{equation}
So, if $t>0$, we write $t=mT_0+r$, where $0\leq r<T_0$, and, by applying \eqref{eq1} repeatedly, we obtain that 
\begin{displaymath}
\vert\det\Phi_t(y)\vert_{L_y}\vert\geq K^{1+\frac{t}{T_0}}e^{\lambda t}= Ke^{\lambda't}.
\end{displaymath}
Since $L_y$ was chosen arbitrarily, we get the desired result.
\end{proof}

\section{Robustly Periodic chain-recurrence classes}\label{section:robustper}

In this section, we establish one of the main results of the paper. More precisely, we prove that for $C^1$-generic vector fields, every non-trivial sectional-hyperbolic chain recurrent class which is not reduced to either a homoclinic loop or a saddle connection between singularities is robustly periodic. This result is stated as Theorem 3.1 below and constitutes a key ingredient in the proof of Theorem A, as it allows us to rule out aperiodic behavior for this chain classes in the generic setting.

\begin{theorem}\label{Teohtop}
There exists a $C^1$-generic set $\mathcal{R}\subset \mathcal{X}^1(M)$ such that if $\Lambda$ is a non-trivial sectional-hyperbolic chain-recurrent class which is not reduced to either a homoclinic loop or a saddle connection between singularities, then it is robustly periodic, that is, there is an open neighborhood $\mathcal{U}$ of $X$ such that for every $Y\in \mathcal{U}$, one has $Per(Y|_{\Lambda_Y})\neq\emptyset$, where $\Lambda_Y$ denotes the continuation of the chain recurrent class $\Lambda$ for $X$. In particular, $\Lambda$ is a homoclinic class.  
\end{theorem}

In order to prove Theorem \ref{Teohtop}, we shall show first that $C^1$-generically every sectional-hyperbolic non-trivial chain-recurrent class which is not reduced to either a homoclinic loop or a saddle connection between singularities actually contains a periodic orbit, and therefore it is a homoclinic class.  As a direct consequence, it has positive topological entropy. 

Recall that Hausdorff distance between two compact subsets $A,B$ of a metric space $(M,d)$ is defined as 
\begin{displaymath}
d_H(A,B)=\max\left\lbrace\sup_{x\in A}d(x,B),\sup_{y\in B}d(y,A)\right\rbrace,\quad d(a,C)=\inf_{c\in C}d(a,c). 
\end{displaymath}
Denote by $\mathcal{K}(M)$ the set of compact sets of $M$. In this case, $(\mathcal{K}(M),d_H)$ is a metric space. 

Now, recall that for a hyperbolic periodic point $p$ associated to a $C^1$ vector field $X$, with splitting $T_{\mathcal{O}(p)}M=E^s\oplus \langle X\rangle\oplus E^u$, the strong stable and  strong unstable manifolds of $p$ are defined, respectively, by 
$$W^{ss}(p)=\{y\in M; \lim\limits_{t\to \infty}d(X_t(p),X_t(y))= 0\}, $$
and
$$W^{uu}(p)=\{y\in M; \lim\limits_{t\to-\infty}d(X_t(p),X_t(y))= 0\}.$$
 Then, the stable and unstable manifolds of $p$, respectively, are given by $$W^s(p)=\bigcup_{t\in \mathbb{R}}W^{ss}(X_t(p))  \textrm{ and }  W^u(p)=\bigcup_{t\in \mathbb{R}}W^{uu}(X_t(p)).$$
These manifolds are tangent to $E^s$, $E^u$, $E^s\oplus \langle X\rangle$, and $E^{u}\oplus \langle X \rangle$, respectively, at each one of the points of $\mathcal{O}(p)$. Thus, the \textit{homoclinic class} of $p$, denoted by $H(p)$, is defined as:
\begin{displaymath}
H(p):=\overline{W^s(p)\pitchfork W^u(p)}.
\end{displaymath}
We say that the homoclinic class $H(p)$ is \textit{non trivial} if it is not reduced to $\mathcal{O}(p)$.

Next, consider the following generic set that is obtained by combining some items from \cite[Proposition 2.1 ]{CY} and \cite[Lemma 3.12]{GYZ}.

\begin{lemma}\label{lemm: genericsets}
There is a residual subset $\mathcal{R}\subset \mathcal{X}^1(M)$ such that for every $X\in  \mathcal{R} $ it holds: 
\begin{enumerate}
    
    \item Any non-trivial chain-recurrent class $C$ of $X$ containing a hyperbolic periodic orbit $\gamma$ is a homoclinic class. 
    \item For every non trivial compact chain-transitive set $C$ of $X$, there is a sequence of periodic orbits $\gamma_n$ of $X$ such that $\gamma_n$ converges to $C$ in the Hausdorff topology.
    \item There is a $C^1$-open neighborhood $\mathcal{U}_X$, if $C(\sigma)$ is a non-trivial chain-recurrent class of $X$ such, then the map $Y\in \mathcal{U}_X\rightarrow C(\sigma_Y)$ is continuous. 
\end{enumerate}
\end{lemma}
The generic property given in the second item of Lemma \ref{lemm: genericsets} was not under consideration in the proof of Theorem A in \cite{CY}. It will be the key to eliminating the assumption of Lyapunov stability in our proofs. 

\begin{remark}\label{remark3.5}
It should be noted that not every non-trivial chain transitive sectional-hyperbolic is encompassed by the set $\mathcal{R}$, as illustrated by the Venice mask. Indeed, the Venice mask is a non transitive sectional-hyperbolic chain-transitive class having a dense set of periodic orbits, so that it cannot be a homoclinic class. The reason is that since the Venice mask is the union non disjoint of two homoclinic classes $H_0$ and $H_1$ whose intersection is the closure of the unstable manifold of a singularity, any periodic of its periodic orbits belongs either $H_0$ or $H_1$, implying that it is not accumulated by the periodic orbits in the Hausdorff topology. 
\end{remark}

Next, we describe the homoclinic structure of $\Lambda$ when it does not contain singularities. 

\begin{lemma}\label{nsyngularcase}
Let $\mathcal{R}$ be the residual set of Lemma \ref{lemm: genericsets}, and let $\Lambda$ be a sectional-hyperbolic chain recurrent class without singularities. Then, it is a homoclinic class.  
\end{lemma}

\begin{proof}
  In this case, $\Lambda$ is hyperbolic by Lemma \ref{hyplemma}. We could now be tempted to conclude the existence of a periodic orbit contained in $\Lambda$ via the Shadowing lemma. Nevertheless, we cannot follow this approach, because we do not know whether $\Lambda$ is isolated or not, and this may prevent the periodic orbits given by the shadowing lemma from being contained in $\Lambda$. Instead, we shall use Lemma \ref{lemm: genericsets} to derive our conclusion.

    First, let $U$ be the neighborhood given by Lemma \ref{SHneighborhood} and let $$\Lambda'=\Lambda\cup \bigcup_{n\geq 1}\gamma_n.$$
\begin{claim}\label{claim1}
$\Lambda'$  is a compact and invariant subset of $U$. 
\end{claim}
\begin{proof}
    First, observe that $\Lambda'\subset U$ and $\Lambda'$ is invariant by definition. Next, let $\lbrace x_k\rbrace_{k\geq 1}\subset \Lambda'$  be a sequence such that $x_k\to x\in M$, when $k\to+\infty$. If  $x_k\in \Lambda$ for sufficiently large $k$, then $x\in \Lambda$, because $\Lambda$ is compact. Similarly, if $x_k\subset \gamma_n$ for sufficiently large $k$ and some $n\geq 1$ fixed, then $x\in \gamma_n\subset\Lambda'$. It remains to verify the case where there are sequences $k_l,n_l\to \infty$, $n_l\neq n_{l'}$, for $l\neq l'$ and $x_{k_l}\subset \gamma_{n_l}$. In this case, since $d^H(\gamma_n,\Lambda)\to 0,$ we have $d(x_{k_l},\Lambda)\to 0$. Hence, $x\in \Lambda\subset\Lambda'$ and the claim holds. 
\end{proof}

Since $\Lambda'\subset U$, Claim \ref{claim1} and the hyperbolic lemma imply that $\Lambda'$ is a compact invariant  hyperbolic set.  Let $T{\Lambda'}=E'^s\oplus \langle X\rangle \oplus  E'^u$ be the hyperbolic splitting of $\Lambda'$. It is well known that dominated splittings vary continuously and hence $\dim(E'^s)$ and $\dim(E'^u)$ are constant along $\Lambda$, because $\Lambda$ is a chain transitive set. Moreover, up to possibly reducing $U$, we can assume $dim(E'^s)$ and $\dim(E'^u)$ are constant along $\Lambda'$  because $d^H(\gamma_n,\Lambda)\to 0$. Also, the continuity of the splitting and the compactness of $\Lambda'$ implies that  the angles between the fibers $E'^s$, $\langle X \rangle$ and $E'^u$ are uniformly limited away from zero.  

   Since $\Lambda'$ is hyperbolic, $W^s_{loc}(x)$ and $W^u_{loc}(x)$ are immersed manifolds, respectively tangent to $E'^s\oplus \langle X\rangle$  and $\langle X \rangle \oplus E'^u$, for any $x\in \Lambda'$. Moreover, those manifolds have a uniform size. For each $n\geq 1$, let $p_n$ be a point of $\gamma_n$. By possibly taking a subsequence, we can assume that $p_n\to x\in \Lambda$. By the uniform size of the local stable and unstable manifolds, the constant dimension of the invariant subspaces and the angle uniformly bounded away from zero, we get $$W^u(p_n)\pitchfork W^s(x) \neq\emptyset \textrm{ and } W^s(p_n)\pitchfork W^u(x)\neq\emptyset,$$
for $n$ sufficiently large. Consequently, $\Lambda$ contains a periodic orbit. Then, the result follows from item (2) of Lemma \ref{lemm: genericsets}. 
\end{proof}

Now, we recall some previous results obtained in \cite{PYY2}, in order to get the proof of Theorem \ref{Teohtop} in the singular case. For this, we follow the exposition given in the aforementioned reference. By Lemma \ref{lemma: hypsing} any singularity $\sigma\in\Lambda$ is hyperbolic. Let consider the hyperbolic splitting
$E^s_\sigma \oplus E^u_\sigma$ of $\sigma$. Without loss of generality, we can think of $\sigma$ as the origin in $\mathbb{R}^n$, and assume that $E^s_\sigma$ and $E^u_\sigma$ are perpendicular (which is possible if one changes the metric). In particular, we will assume that $E^s_\sigma = \mathbb{R}^s$ is the $s$-dimensional subspace of $\mathbb{R}^n$ with the last $n-s$ coordinates being zero. Here $s = \dim E^s_\sigma$ is the stable index of $\sigma$. Similarly, $E^u_\sigma$ is the subspace of $\mathbb{R}^n$ where the first $s$ coordinates are zero. 


 Since the vector field $X$ is $C^1$, we can take a neighborhood $V_{\sigma}= B_r(\sigma)$ with $r$ small enough, such that the flow in $V_{\sigma}$ can be written as
     \begin{displaymath}
     X_t(x) = e^{At}x + C^1 \text{ small perturbation}, 
     \end{displaymath}
 where $A$ is a matrix with no eigenvalue on the imaginary axis. 

For each $x \in V_{\sigma}$, denote by $x^s$ its distance to $E^s_\sigma$ and $x^u$ its distance to $E^u_\sigma$. Then for every $\alpha > 0$ small, we define the $\alpha$-cone on the manifold, denote by $D^i_\alpha(\sigma)$, $i = s,u$, as follows:
\[
D^s_\alpha(\sigma) = \{x \in V_{\sigma} : x^u < \alpha x^s\}, \qquad
D^u_\alpha(\sigma) = \{x \in V_{\sigma} : x^s < \alpha x^u\}.
\]

Note that the hyperbolic splitting $T_\sigma M = E^s_\sigma \oplus E^u_\sigma$ can be extended to $V_{\sigma}$ in a natural way: for each $x \in V_{\sigma}$, put $E^s(x)$ as the $s$-dimensional hyperplane that is parallel to $E^s_\sigma$; the same can be done for $E^u(x)$. This allows us to consider the $\alpha$-cones $C_\alpha(E^i)$, $i = s,u$, on the tangent bundle of $V_{\sigma}$. 

Let us fix some $\alpha > 0$ small enough. One can think of the region $V_{\sigma} \setminus (D^s_\alpha(\sigma) \cup D^u_\alpha(\sigma))$ as the place where the flow
is ``making the turn'' from the $E^s$-cone to the $E^u$-cone. For each $x \in U$, define 
\[
t^+(x) = \sup\{t > 0 : X_{[0,t]}(x) \subset V_{\sigma}\}, \qquad
t^-(x) = \sup\{t > 0 : X_{[-t,0]}(x) \subset V_{\sigma}\}.
\]
It is easy to check that the orbit segment $X_{(-t^-,t^+)}(x)$ contains $x$ and it is contained in $U$. In this way, we write
 \begin{equation}\label{entradasalida}
 x_{e} = X_{-t_i^-}(x), \qquad x_{o} = X_{t_i^+}(x)
 \end{equation}


\begin{lemma}[Lemma 3.2 of \cite{PYY2}]\label{lemma3.2}
Let $\sigma$ be a hyperbolic singularity for a $C^1$ vector field $X$. Then for every $\alpha > 0$ small enough,
there exists $T_\alpha > 0$ such that for every $r > 0$ small enough and every
$x \in V_{\sigma}=B_r(\sigma)$, the set
\[
T(x) := \{t \in (-t^-,t^+) : X_t(x) \notin D^s_\alpha(\sigma) \cup D^u_\alpha(\sigma)\}
\]
has length smaller than $T_\alpha$.
\end{lemma}

Now, for a non singular point $x\in M$, the \textit{orthogonal complement} $N_x$ of $X(x)$ is defined by  
\begin{displaymath}
N_x=\lbrace v\in T_xM : \langle v,X(x)\rangle=0\rbrace. 
\end{displaymath}
In this way, the \textit{linear Poincar\'e flow} $\psi_t$ on the normal bundle $N_{M\setminus Sing(X)}$ is given by 
\begin{displaymath}
\psi_t(v)=\pi_{N_{X_t(x)}}(\Phi_t(x)v),\quad \forall x\in M\setminus Sing(X), v\in N_x, 
\end{displaymath}
where $\pi_{N_z}(w)$ is the orthogonal projection of $w$ on $N_z$. 
The \textit{scaled linear Poincar\'e flow}, denoted by $\psi^*_t$, is defined as the normalization of $\psi_t$ by the flow speed, i.e., 
\begin{displaymath}
\psi^*_t(v)=\frac{\Vert X(x)\Vert}{\Vert X(X_t(x))\Vert}\psi_t(v).    
\end{displaymath}

In \cite{BM} it is showed that a partially hyperbolic splitting $T_{\Lambda}M=E\oplus F$ of a compact invariant set $\Lambda$ induces a partially hyperbolic splitting  $T_xN_x=\mathcal{E}_x\oplus \mathcal{F}_x$ of the normal subspace $N_x$ for any $x\in\Lambda\setminus Sing(X)$, where  $\mathcal{E}_x=\pi(E_x)$ and $\mathcal{F}_x=\pi_x(F_x)$ for every $x\in\Lambda$, where $\pi$ denotes the orthogonal projection. Indeed, since $F_x$ contains the flow direction, it follows that $\pi(v)\in F_x$ for any $v\neq X(x)$ in $F_x$, hence $\mathcal{F}_x\subset F_x$ because $\pi$ is linear. So, since the splitting $E\oplus F$ is dominated for $\Phi_t$, we deduce the dominance of the splitting $\mathcal{E}\oplus \mathcal{F}$. The uniform contracting property for $\pi(E)$ follows because $\Vert\pi(v)\Vert\leq\Vert v\Vert$ for every $v\in E$. It is important to observe that the above property remains valid for the rescaled linear Poincar\'e flow $\psi^*_t$.  

Recall the notion of $(\eta, T, \mathcal{F})$-backward contracting points. 
\begin{definition}
Let $X\in X^1(M)$, $\Lambda$ a compact invariant set of $X$, and $\mathcal{F}\subset N_{\Lambda\setminus Sing(X)}$ an invariant bundle of the rescaled linear Poincar\'e flow $\psi^*_t$. For $0<\eta<1$ and $T_0>0$, we say that an orbit segment $X_{[-T,0]}(x)$, $x\in \Lambda\setminus Sing(X)$, is \textit{$(\eta,T_0)$-backward contracting for the bundle $\mathcal{F}$} if there is a partition $0=t_0<t_1<t_2<\cdots<t_k=T$, where $t_{i+1}-t_i\in[T_0,2T_0]$ for every $i=0,1,\ldots, k-1$ such that 
\begin{displaymath}
\prod_{i=0}^{m-1}\Vert\psi^*_{-(t_{i+1}-t_{i})}(X_{t_{i+1}}(x))\vert_{\mathcal{F}(X_{t_{i+1}}(x))}\Vert\leq \eta^{m}.
\end{displaymath}
On the other hand, we say that $x\in \Lambda\setminus Sing(X)$ is a \textit{$(\eta,T_0)$-backward hyperbolic point for the bundle $\mathcal{F}$} if there is a partition $0=t_0<t_1<t_2<\cdots<t_k<\cdots$, where $t_{i+1}-t_i\in[T_0,2T_0]$ for every $i=0,1,\ldots$ such that 
\begin{displaymath}
\prod_{i=0}^{k-1}\Vert\psi^*_{-(t_{i+1}-t_{i})}(X_{t_{i+1}}(x))\vert_{\mathcal{F}(X_{t_{i+1}}(x))}\Vert\leq \eta^{k},\quad\forall k\geq 1.
\end{displaymath}  
\end{definition}


The next lemma explain the behavior of the orbit segments associated to points in the cones $D_{\alpha}^*(\sigma)$, $*=s,u$, for every $\sigma\in Sing(X|_{\Lambda})$, with respect to the scaled LPF. 

\begin{lemma}[Lemma 3.9 of \cite{PYY2}]\label{bacward}
Let $\sigma$ be a Lorenz-like singularity, then there exists $\eta\in(0,1)$, $T_0>0$, $\alpha_1>0$ and $r>0$ such that if $\alpha<\alpha_1$ and $x$ is a periodic orbit such that the orbit segment $X_{[0,T]}(x)$ is contained in $B_r(\sigma)\cap \overline{D_{\alpha}^s(\sigma)}$, then the orbit segment is
$(\eta, T_0)$-backward contracting. If the orbit segment is in $\overline{D_{\alpha}^u(\sigma)}$, then it does not
have any sub-segment that is backward contracting.
\end{lemma}

Finally, we present the next lemma, which allows us to find a size of the cone field on the $\alpha$-cone on the manifold in order to get transverse intersections between invariant manifolds.

\begin{lemma}[Lemma 3.11 of\cite{PYY2}]\label{intersections}
Let $\sigma$ be a Lorenz-like singularity. Then for $r > 0$ small enough, for every $\eta \in (0,1)$,
$T_0 > 0$ there exists $\alpha_2 > 0$ such that for all $\alpha < \alpha_2$, if
$y \in B_r(\sigma) \cap Cl(D^c_\alpha(\sigma))$ is a $(\eta,T_0)$-backward hyperbolic point on its orbit,
then
\[
W^u(y) \cap W^{cs}(\sigma) \neq \varnothing.
\]
\end{lemma}

\begin{remark}\label{chosealpha}
In our case, since $Sing(X)_{\Lambda}$ is finite, the positive nummber $$\alpha=\min_{\sigma\in Sing(X|_{\Lambda})}\{\alpha_1(\sigma),\alpha_2(\sigma)\})>0$$ satisfies the conclusion of Lemma \ref{bacward} and Lemma \ref{intersections} for every $\sigma\in Sing(X|_{\Lambda})$.      
\end{remark}

We now present the following lemma, which will be used to prove Theorem \ref{Teohtop}.
\begin{lemma}\label{prop2.4CY}
Let $\mathcal{R}$ be the residual set of Lemma \ref{lemm: genericsets}. Let $X\in \mathcal{R}$, $\sigma_X\in Sing(X)$ and $C(\sigma)$ be a non-trivial sectional-hyperbolic chain-recurrent class. There is a neighborhood $\mathcal{U}_0\subset\mathcal{U}_X$ of $X$ such that for every $Y\in \mathcal{U}_0$, the continuation $C(\sigma_Y)$ is sectional-hyperbolic. 
\end{lemma}
\begin{proof}
   Let $X\in \mathcal{R}$ and $\Lambda$ a non-trivial sectional-hyperbolic class of $X$.  Recall that sectional-hyperbolicity is a robust property, i.e., there are neighborhoods $U$ of $\Lambda$ and $\mathcal{U}'_X$ of $X$ such that for every $Y\in \mathcal{U}'_X$ the set $\cap_{t\in\R}Y_t(U)$ is sectional-hyperbolic. Let   $\mathcal{U}_X$ be given by item 3 of Lemma \ref{lemm: genericsets}. Denote $\mathcal{U}_0=\mathcal{U}_X\cap\mathcal{U}'_X$. Since the map $Y\in \mathcal{U}_X\to C(\sigma)$ is continuous, after possibly shrinking $\mathcal{U}_X$, we can assume $C(\sigma_Y)\subset U$. Therefore, $C(\sigma_Y)$ is contained in the maximal invariant set of $U$ under $Y$. In particular, $C(\sigma_Y)$ is sectional-hyperbolic.  
\end{proof}

\begin{remark}
    By the above lemma, we see that Lemma \ref{lemma: hypsing} holds for every $Y\in\mathcal{U}_0$. 
\end{remark}

In order to prove Theorem \ref{Teohtop} in the singular case, we divided the argument into two steps: 
\begin{enumerate}
    \item To find a transverse intersection between the stable manifold of $\sigma_Y\in \Lambda_Y$ and the unstable manifold of some periodic orbit $\gamma_Y$ for $Y$. 
    \item To find a transverse intersection between the stable manifold of a periodic orbit $\gamma_Y$ for $Y$ (whose dimension is $\dim E^s$) and some disk $D^c$ of dimension $\dim F$, tangent to the central cone. 
\end{enumerate}

The next result gives us step 1 in our argument. 

\begin{lemma}\label{connection1}
Let $X\in \mathcal{R}$ and let $C(\sigma_X)\subset M$ be a non-trivial sectional-hyperbolic chain-recurrent class for $X$. Then, there is a neighborhood $\mathcal{U}_1\subset\mathcal{U}_0$ of $X$ such that for every $Y\in\mathcal{U}_1$, 
\begin{displaymath}
    W^s(\sigma_Y)\pitchfork W^u(\gamma_Y)\neq\emptyset,
\end{displaymath}
for some periodic orbit $\gamma_Y\in Per(Y)$. 
\end{lemma}

\begin{proof}
To begin with, let $U$ be the neighborhood of $\Lambda$ given by Lemma \ref{SHneighborhood}. By Lemma \ref{lemm: genericsets} there is a sequence of periodic orbits of $\gamma_n\subset M$ converging to $\Lambda$ in the Hausdorff metric. So, we can assume $\gamma_n\subset U$ for every $n\geq 1$. Denote $$\Lambda'=\Lambda\cup \bigcup_{n\geq 1}\gamma_n.$$
 Notice that by Claim \ref{claim1} of Lemma \ref{nsyngularcase} and Lemma \ref{SHneighborhood} the set $\Lambda'$ is sectional hyperbolic.

 Now we assume that $\Lambda=C(\sigma_X)$, for some singularity $\sigma_X\in\Lambda$. Consider the  neighborhood $U$ from Lemma \ref{SHneighborhood}  and  $\Lambda'$. Let $T\Lambda'=E\oplus F$ be the sectional-hyperbolic splitting of $\Lambda'$. Note that the continuity of the splitting implies that $\dim(E)$ and $\dim(F)$ are constant along $\Lambda'$ and that the angles between $E$ and $F$ are uniformly bounded away from zero. Since $E$ is uniformly contracted, by the classical stable manifold theory, for every $x\in \Lambda'$  there is a well defined local stable manifold $W^{ss}_{loc}(x)$ with uniform size along $\Lambda'$.

Since $\Lambda'$ is not hyperbolic, there may not be well defined for local unstable manifolds for every point in $\Lambda'$. Moreover, even though these local unstable manifolds may exist for some points (e.g. along non-singular compact and invariant subsets) they do not need to have uniform size. To overcome this issue, we firstly recall that the periodic orbits $\gamma_n$ are hyperbolic by the hyperbolic lemma. So, fix $\gamma_n$ and let $$T\gamma_n=E'^s\oplus \langle X\rangle \oplus E'^u$$ be the hyperbolic splitting of $\gamma_n$. Since $E$ is uniformly contracted, we have $E(x)\subset E'^s(x)$, for every $x\in \gamma_n$. Now, suppose there is $x\in \gamma_n$ so that $E'^s(x)\not \subset E(x)$. In this case, there is a vector $v\in E'^s(x)$ such that $v=v^E + v^F$, $v^E\in E(x)$, $v^F\in F(x)$ and $v^F\neq 0$.  Moreover, since $v^F\in E'^s(x)$, then $v\notin\langle X(x)\rangle$. On the other hand,  by considering the subspace $L\subset F$ generated by $v^F$ and $\langle X(x)\rangle$, the sectional-expansion of $L$ implies,  $v^F\notin E'^s(x)$. This contradiction implies that $E(x)=E'^s(x)$ for every $x\in \gamma_n$. Consequently, $$\langle X(x) \rangle \oplus E'^u(x)=F(x),\quad \forall x\in \gamma_n.$$    

Choose $r>0$ as in Lemma \ref{intersections} such that  $B_{r}(\sigma_X)\cap Sing(X)=\{\sigma_X\}$. Since any singularity contained in $C(\sigma_X)$ is hyperbolic by Lemma \ref{lemma: hypsing}, there are only finitely many of them.  By shrinking $U$ if it is necessary, we have that $Sing(X|_U)=Sing(X|_{C(\sigma_X)})$ and hence, there is a neighborhood $V$ of $Sing(X|_{C(\sigma_X)})$ contained in $U$  satisfying  $$W^s_{loc}(\sigma)\not \subset V \textrm{ and } W^u_{loc}(\sigma)\not \subset V.$$ 
    Moreover, since $Sing(X|_{V})$ is finite, we can assume there is a  family of pairwise disjoint open sets $V_1,..., V_{j}$ such that:
    \begin{itemize}
    \item $V_i=B_r(\sigma_j)$, where $\lbrace\sigma_j : j=1,\ldots, k\rbrace=Sing(X|_{C(\sigma_X)})$. 
    \item $V=V_1\cup\cdots\cup V_j$.
    \item Each $V_i$ is a neighborhood of a unique singularity in $X|_{U}$.
    \end{itemize}

Fix $T_0>0$ in a such way that 
\begin{equation}\label{chooooseT0}
    Ke^{-\lambda t}\leq e^{-\lambda' t}\text{ and }K_0^2e^{-\lambda't}\leq e^{-\lambda'' t}, \quad\forall t\geq T_0,
\end{equation}
where $K,\lambda>0$ the constants given by sectional hyperbolicity of $\Lambda'$, $0<\lambda''<\lambda'<\lambda$ and $$K_0=\sup\left\{\frac{\Vert X(x) \Vert}{\Vert X(y) \Vert} :  x,y\in C(\sigma_X)\setminus V\right\}.$$ 

For any point $q\in V$, let $q_e$ be given as \eqref{entradasalida}. Suppose that $\phi_{[-T,0]}(q^n_{e})\subset C(\sigma_X)\setminus V$, where $T>T_0$. Define $\eta_0=e^{-\lambda''T_0}$.  

\begin{claim}
The orbit segment $\phi_{[-T,0]}(q_{e})$ is $(\eta_0,T_0)$-backward contracting for the bundle $\mathcal{F}$.     
\end{claim}
\begin{proof}
Let $0=t_0<t_1<\cdots<t_m=T$ such that $t_{j+1}-t_j=T_0$ for every $i=0,\ldots,m-2$ and $t_m-t_{m-1}\in(T_0,2T_0)$. Notice that the orbit segment $\phi_{[-T,0]}(q_{e})\subset C(\sigma_X)\setminus V$. Then, for every unitary vector $v\in\mathcal{F}_{q_{e}}$ one has 
\begin{eqnarray*}
    Ke^{-\lambda(t_{j+1}-t_j)}&\geq&\vert\det\Phi_{-(t_{j+1}-t_j)}(\phi_{t_{-j}}(q_{e}))\vert\\
    &\geq&\frac{\Vert X(X_{-t_{j}}(q_{e}))\Vert}{\Vert X(X_{-t_{j+1}}(q_{e}))\Vert}\cdot\Vert\psi_{-(t_{j+1}-t_j)}(\phi_{t_{-j}}(q_{e}))v_{j}\Vert,
\end{eqnarray*}
where $v_1=v$ and $v_{j+1}=\psi_{-(t_{j+1}-t_j)}(X_{-t_{j}}(q_n))v_{j}$ for any $j=1,\ldots, m-1$, and so, by \eqref{chooooseT0},
\begin{displaymath}
 \Vert\psi_{-(t_{j+1}-t_j)}(X_{-t_{j}}(q_{e}))\vert_{\mathcal{F}_{X_{-t_{j}}(q_{e})}}\Vert\leq \frac{\Vert X(X_{-t_{j+1}}(q_{e}))\Vert}{\Vert X(X_{-t_{j}}(q_{e}))\Vert}e^{-\lambda'(t_{j+1}-t_j)},\quad \forall j=0,\ldots, m-1.
\end{displaymath}

Thus, we have by the choice of $K_0$ in \eqref{chooooseT0} that 
\begin{eqnarray*}
 \prod_{j=0}^{m-1}\Vert\psi^*_{-(t_{j+1}-t_i)}(X_{-t_{j}}(q_{e}))\vert_{\mathcal{F}_{X_{-t_{j}}(q_{e})}}\Vert&\leq&\prod_{j=0}^{m-1}\left[\frac{\Vert X(X_{-t_{j+1}}(q_{e}))\Vert}{\Vert X(X_{-t_{j}}(q_{e}))\Vert}\right]^2e^{-\lambda'(t_{i+1}-t_i)}\\
 &\leq&\left[\frac{\Vert X(X_{-t_1}(q_{e}))\Vert}{\Vert X(q_e)\Vert}\right]^2e^{-m\lambda'T_0}\\
 &\leq&\left(K_0\right)^2e^{-m\lambda'T_0}\\
 &\leq& (e^{-\lambda''T_0})^m=\eta_0^m.
\end{eqnarray*}
This proves the claim.
\end{proof}
    
Notice that if $T_0>T_1>0$ and $0<\lambda_1<\lambda''$, then  
\begin{eqnarray}\label{implica}
    \lambda''T_0=\lambda''T_1\cdot\frac{T_0}{T_1}>\lambda''T_1
    &\Rightarrow& -\lambda''T_0<-\lambda''T_1<-\lambda_1T_1 \\ 
    &\Rightarrow& \eta_0=e^{-\lambda''T_0}<e^{-\lambda''T_1}<e^{-\lambda_1T_1}=\eta'. \nonumber
\end{eqnarray}
Hence, if the orbit segment $\phi_{[-T,0]}(q_{e})$ is $(\eta_0,T_0)$-backward contracting for the bundle $\mathcal{F}$, then it is also $(\eta_1,T_1)$-backward contracting for the same bundle $\mathcal{F}$ for every $0<T_1<T_0$ and $\eta\in(\eta',1)$.   

Since $\Pi(\gamma_n)\to+\infty$, we have by te Pliss Lemma that there are $\eta,\eta_1,T_1>0$ and $p_n\in\gamma_n$ such that $p_n$ is a $(\eta,T_1)$-backward hyperbolic point for $\mathcal{F}$ along the orbit of $\gamma_n$. Moreover, such points have positive density $a>0$ along the orbit of $\gamma_n$ for every $n\geq 1$. In this case, we have the following possibilities: 
\begin{itemize}
   \item $\sup\lbrace T>0 : X_{[-T,0]}(q_{e,n})\subset C(\sigma_X)\setminus V\rbrace< T_0$ for every $n\geq 1$. Since the density of $(\eta,T_1)$-backward hyperbolic times is $a>0$ it follows that there must be hyperbolic times $p'_n\in\gamma_n\cap V_{\sigma_X}$. In particular, there is $s_n\in\mathbb{R}$ such that $X_{s_n}(p_n)\in\overline{D^s_{\alpha}(\sigma_X)}$, or 
    \item there is a subsequence $\lbrace p_{n_k}\rbrace_{k\geq 1}$ such that $$\sup\lbrace T>0 : X_{[-T,0]}(q_{e,n_k})\subset C(\sigma_X)\setminus V\rbrace\geq T_0,\quad \forall k\geq 1.$$ In this case, if $p_{n_k}\in V_{\sigma_X}$, we proceed as item 1 above. If $p_{n_k}\notin V_{\sigma_X}$, then by the Claim 2 and \eqref{implica}, there is $s>0$ such that $X_s(p_{n_k})=q^{n_k}_{e}\in \overline{D^s_{\alpha}(\sigma)}$ is also a $(\eta,T_1)$-backward hyperbolic point for the bundle $\mathcal{F}$ for every $k\geq 1$. 
\end{itemize}

So, it follows from the both cases and by taking a sequence if it is necessary, that $q_n\in\gamma_n$ is a $(\eta,T_1)$ backward hyperbolic points inside $\overline{D^s_{\alpha}(\sigma)}$ for any $n\geq 1$. Therefore, by Lemma \ref{intersections} we have that 
 \[
W^u(p_n) \cap W^{s}(\sigma) \neq \varnothing.
\] 

Now, since $X\in\mathcal{R}$, by item (3) of Lemma \ref{lemm: genericsets} and Lemma \ref{SHneighborhood} we can find a neighborhood $U$ of $C(\sigma_X)$ and a neighborhood $\mathcal{U}_X$ of $X$ such that the map $Y\in \mathcal{U}_X\to C(\gamma_Y)$ is continuous and the maximal invariant set $\Lambda_Y$ for $Y\in \mathcal{U}_X$ in $U$ is sectional-hyperbolic for $Y$ and it contains $C(\sigma_Y)$.
In particular, there are $\delta_0\approx\delta$ and a neighborhood $\mathcal{U}_1\subset\mathcal{U}_X$ of $X$ such that $W^s(\sigma_Y)\pitchfork W_{\delta_0}^u(\gamma_Y)\neq\emptyset$ for any $Y\in \mathcal{U}_1$. This proves the result. 
\end{proof}

Let $\Lambda=C(\sigma_X)$ and $\Lambda_Y=C(\sigma_Y)$. Fix a neighborhood $U$ for $\Lambda$ and $\Lambda_Y$, given by the continuity of the map $Y\to C(\gamma_Y)$. Since $\Lambda$ is partially hyperbolic, it is possible to extend the sectional-hyperbolic  splitting $T_{\Lambda}M$ to $U$ (by shrinking $U$ if it is possible), that we will keep written as $T_UM=E\oplus F$. It should be noted that since $\Lambda$ is sectional-hyperbolic one has $1\leq \dim E\leq n-2$ and $1\leq \dim F\leq n-1$. Thus, the subspace $N_x$ is of codimension one and has a coordinate system induced by the sectional-hyperbolic splitting for every $x\in U$ in the following way: 
\begin{displaymath}
    \mathcal{E}_x=\pi(E_x)\text{ and }\mathcal{F}_x=\pi(F)\text{ or }F_x\cap N_x,\quad x\in U, 
\end{displaymath}
From partial hyperbolicity we have that the angle between these subspaces is uniformly bounded away from zero. Denote by $N(\eta)$ the standard cube $(-\eta,\eta)^{n-1}\subset N_x$, according to this coordinate system.

For any regular point $x\in U$, there is a number $\eta(x)>0$ such that the exponential map $Exp_x:N_x\to M$ is an isometry. Note that for every $\delta>0$, there is $\eta_{\delta}>0$ such that $\eta(x)>\eta_{\delta}$ for any $x\in U\setminus B(Sing(\Lambda),\delta)$. Let $\eta^*:U\to \mathbb{R}^+$ defined by $\eta^*(x)=\min\lbrace \eta(x),\eta_{\delta}\rbrace$ for every $x\in U$. Denote by $\mathcal{N}_x=Exp_x(N_x(\eta^*(x)))\subset M$, a codimension 1 cross section on $x$ of size $\eta^*(x)$. 

For a regular point $x$, the coordinates $\mathcal{E}_x$ and $\mathcal{F}_x$ allows us to define the volume of boxes $B\subset N_x$ with respect to these axis as follows: Denote by $\hat{\pi}_s:N_x\to\mathcal{E}_x$ and $\hat{\pi}_c:N_x\to\mathcal{F}_x$ the projections along the $\mathcal{E}_x$-axis and the $\mathcal{F}_x$-axis respectively. The stable (central) volume of $B$ is given by 
\begin{displaymath}
    \vert B\vert_{s(c)}:=vol_{s(c)}(\pi_{s(c)}(B)), 
\end{displaymath}
where $vol_{s(c)}(A)$ denotes the volume of a subset $A\subset \mathbb{R}^{\dim \mathcal{E}_x (\mathcal{F}_x)}$. 
We shall use a vector $v=(\varepsilon^s_1,\varepsilon^s_2,\ldots,\varepsilon^s_{n_s},\varepsilon^c_1, \varepsilon^s_2,\ldots,\varepsilon^c_{n_c})$, where $1\leq n_s\leq n-2$ and $1< n_c\leq n-1$ satisfy $n_s+n_c+1=n$, to describe the box $B[v]$ with stable volume $\varepsilon_s$ and central volume $\varepsilon_c$. 

Let $x\in U$ be a regular point does not belong to the stable manifold of the singularities of $X$, and $t>0$. By the Implicit Function Theorem that there is an open connected subset $D\subset\mathcal{N}_x$ containing $x$ and a continuous function $\tau:D\to N$ such that $\tau(x)=t$ and $X_{\tau(y)}(y)\in \mathcal{N}_{X_t(x)}$, for any $y\in D$. The map $\tau$ induces a diffeomorphism 
\begin{displaymath}
    H_x^t:dom(H_x^t)\subset \mathcal{N}_x\to\mathcal{N}_{X_t(x)},
\end{displaymath}
where $dom(H_x^t)$ is the maximal connected component of $\mathcal{N}_x$, containing $x$ in its interior, where $\tau$ is defined. the map $H_x^t$ is called \textit{the holonomy map between $\mathcal{N}_x$ and $\mathcal{N}_{X_t(x)}$ of time $t>0$}. 

Next we present, in a precise way, the step 2 in our argument. 

\begin{lemma}\label{connection2}
Let $X\in \mathcal{R}$ and let $C(\sigma_X)\subset M$ be a non-trivial sectional-hyperbolic chain-recurrent class for $X$, which is not reduced to either a homoclinic loop or a saddle connection between singularities. There is a neighborhood $\mathcal{U}\subset \mathcal{U}_2$ of $X$ such that,  for every $Y\in\mathcal{U}$, the periodic orbit $\gamma_Y$ given by Lemma \ref{connection1} is chain-attainable from $\sigma_Y$.
\end{lemma}

\begin{proof}
 First, we recall that by Lemma \ref{lemma: hypsing} the singularities of $C(\sigma_X)$ are Lorenz-like. Let us denote $$Sing(X|_{C(\sigma_X)})=\{\sigma_X=\sigma^0,\sigma^1,...,\sigma^k\}.$$   
 
 Fix  $\delta>0$ so that $$\Vert X(z)\Vert\leq \Vert X(w)\Vert, \quad\forall z\in B(Sing(X|_{C(\sigma_x)}),\delta), w\notin B(Sing(X|_{C(\sigma_x)}),\delta).$$ Also, from item (2) of  Lemma \ref{lemm: genericsets} there is a sequence of periodic orbits $\lbrace\gamma_n\rbrace_{n\geq 1}$ of $X$ such that $d_H(\gamma_n,C(\sigma_X))\to 0$. In this case, let us assume $$d_H(\gamma_n,C(\sigma_X))<\delta/4,$$ for every $n\geq 0$. In particular, for each $i=0,,...,k$,  we can obtain sequences of points  $(p_{n,i}^s)_{n\geq 1}$ and $(p_{n,i}^u)_{n\geq1}$ such that  for every $i=0,...,k$:
\begin{itemize}
    \item $p_{n,i}^s, p_{n,i}^u$ are contained in $\gamma_n$, for every $n\geq 1$. 
    \item $p^u_{n,i}\to p_{u,i}\in C(\sigma_X)\cap W^u(\sigma^i)$,
    \item$p^s_{n,i}\to p_{s,i}\in C(\sigma_X)\cap W^s(\sigma^i)$,
    \item $d(p_{s,i},\sigma_i), d(p_{u,i},\sigma_i)=\delta$,
\end{itemize}

Next, for such a value of $\delta>0$, take a family of cross-sections $$\mathcal{N}_i^{in}=\mathcal{N}^{in}_{p_{s,i}}\subset \partial B(\sigma_i,\delta) \textrm{ and } \mathcal{N}_i^{out}=\mathcal{N}^{out}_{p_{u,i}}\subset \partial B(\sigma_i,\delta)$$ at $p_{s,i}$ and $p_{u,i}$ respectively. Note that the size $\eta$ of both $\mathcal{N}_i^{in}$ and $\mathcal{N}_i^{out}$ is bigger or equal than $\eta_{\delta}$. 
Moreover, for any $i=0,...,k$, the holonomy map $$H_i:Dom(H_i)\subset \mathcal{N}^{in}_i\setminus W^s_{loc}(\sigma^i)\to \mathcal{N}^{out}_i$$ is well defined. Since $p^{s}_{n,i}\to p_{s,i}$ and $p^{u}_{n,i}\to p_{u,i}$, we can  assume, without loss of generality, that $p^{s}_{n,i}\in \mathcal{N}_i^{in}$ and $p^{u}_{n,i}\in \mathcal{N}_i^{out}$, for every $n\geq 0$. Take $N\in\mathbb{N}$ satisfying  $$d^i_s=d(Q^s_{X,i},p^s_{N,i})<\eta/2 \textrm{ and } d^i_u=d(Q^u_{X,i},p^u_{N,i})<\eta/2,$$ where $$Q^{s}_{X,i}=\mathcal{N}^{in}_i\cap W^{s}(\sigma^i) \textrm{ and } Q^{u}_{X,i}=\mathcal{N}^{out}_i\cap W^{u}(\sigma^i).$$ 

Now, let consider the following cases:

\vspace{0.1in}
\textbf{Case 1: }There is $x\in C(\sigma_X)$ such that $\omega(x)\cap Sing(X)=\emptyset$.
\vspace{0.1in}
 
Since $\omega(x)$ has no singularities, it follows by Lemma \ref{hyplemma} that it is hyperbolic of saddle type. In particular, it has local stable and unstable manifolds of size $\varepsilon>0$. So, we obtain that $$W^u(z)\pitchfork W^s(\gamma_X)\neq\emptyset$$ for some $z\in\omega(x)\subset C(\sigma_X)$. Thus, since $z$ is chain attainable from $\sigma_X$, one has that $\gamma_X$ is chain attainable from $\sigma_X$. 

\vspace{0.1in}
\textbf{Case 2: }$\omega(x)\cap Sing(X)\neq\emptyset$ for any $x\in C(\sigma_X)$. 
\vspace{0.1in}

In this case, we denote $$\mathcal{N}^{in}=\bigcup_{i=0}^k\mathcal{N}_{i}^{in} \textrm{ and } \mathcal{N}^{out}=\bigcup_{i=0}^k\mathcal{N}_{i}^{out}.$$
The compactness of $C(\sigma_X)\setminus B(Sing(X|_{C(\sigma_X)},\delta)$ implies the existence of a positive number $T_{\delta}$ such that if $$z\in C(\sigma_X)\setminus B(Sing(X|_{C(\sigma_X)},\delta),$$ there is $0<t_z\leq T_{\delta}$  such that $X_{t_z}(z)\in \mathcal{N}^{in}$ and $X_{(0,t_z)}(z)\cap \mathcal{N}^{in}=\emptyset$.

By hypothesis on $C(\sigma_X)$, there is a point $x\in C(\sigma_X)\setminus W^s(Sing(X))$. Moreover, since the positive orbit of $x$ accumulates $Sing(X\vert_{C(\sigma_Y)})$, there is a sequence of times $t_n\to \infty$ and a singularity $\sigma^i\in Sing(X|_{C(X)})$ such that $X_{t_n}(x)\to \sigma^i$. In particular, the positive orbit of $x$ crosses $\mathcal{N}^{in}$ infinitely many times. Thus, we may assume  $x\in \mathcal{N}^{in}_i$ for some $i\in\lbrace 0,1,\ldots, k\rbrace$. Let us consider a box $B\subset \mathcal{N}^{in}_i$ containing $x$ with central volume $0<b<\frac{d^{s}_Y}{2}$.

\vspace{0.1in}
\textbf{Claim:} For any $T>0$, there is $T_b>T$, depending on $b$, such that $$X_{[0,T_{b}]}(B)\pitchfork W^s(\gamma_X)\neq\emptyset.$$

Indeed, let consider a unitary vector $\hat{v}\in\mathcal{F}_x$, and take the plane $L_x$ spanned by $\hat{v}$ and $Y(x)$. Note that $\hat{v}$ is either contained in $F_x$ or has the form $\hat{v}=\pi(v)$ for some $v\in F_x$. So, by denoting  $$K=\inf\left\{\frac{\Vert X(z)\Vert}{\Vert X(w)\Vert} :  z,w\notin B(\sigma_X,\delta)\right\},$$
the sectional-hyperbolicity implies  
\begin{eqnarray}\label{explane}
\Vert \pi_{N_{X_t(x)}}(\Phi_{t}(\hat{v}))\Vert&=&\frac{\Vert X(x)\Vert}{\Vert X(X_{t}(x))\Vert}\cdot\det(\Phi^X_{t}(x)\vert_{L_x}) \nonumber \\
    &\geq&\frac{\Vert X(x)\Vert}{\Vert X(X_{t}(x))\Vert}\cdot e^{\lambda t} \nonumber\\
    &\geq& Ke^{\lambda t},
\end{eqnarray}
whenever $t>0$, $X_{[0,t)}(x)\in U$ and $Y_t(x)\notin B(\sigma,\delta)$. 

Recall by the construction of the cross sections $\mathcal{N}^{in}$ and $\mathcal{N}^{out}$ that if $z\in \mathcal{N}^{in}_i$, then  $X_{[0,\tau(z))}(z)\in B(\sigma_X,\delta)\subset U$ and   $H_i(z)=X_{\tau(z)}(z)\in\mathcal{N}^{out}_i$. In this case, define $$\tau_i=\inf\{\tau(z) : z\in \mathcal{N}^{in}_i\}.$$ Notice that we can choose $\delta>0$ in a such way that $$Ke^{\lambda t}\geq e^{\lambda_ut},\quad\forall i=0,1,\ldots, k,\, \forall t>\tau_i,$$ where $0<\lambda_u<\lambda$.
Let $\mathcal{N}_{X_t(x)}$, $t>\tau(x)$, be a cross-section through $X_t(x)$ of radius at least $\eta_\delta$. Since $C(\sigma_Y)$ is not necessarily Lyapunov stable, the set $H^t_x(B)$ may not be entirely contained in $B(C(\sigma_X),2\delta)$ for some $t>\tau_i$. Nevertheless,  there is a connected component $B_1$ of $H_x^{t}(B)$ whose central volume is bigger than $b$ and $\overline{B}\subset B(C(\sigma_X),2\delta)$. In this case, we set $$r_1=\min\lbrace r>0 : X_r(z)\in \partial B(C(\sigma_X),2\delta)\text{ for some }z\in (H^r_x)^{-1}(B_1)\rbrace.$$  
Then, if $t=\tau(x)+r_1$,by shrinking $\mathcal{N}^{in}$ if it is necessary we have by \eqref{explane} that $$\Vert \pi_{N_{X_{\tau(x)+r}(x)}}(\Phi_{\tau(x)+r_1}(\hat{v}))\Vert\geq e^{\lambda_u(\tau+r_1)},\quad \forall \hat{v}\in \hat{F}_x.$$ Hence, 
the map $$H^{r_1}_x:Dom(H^{r_1}_x)\subset \mathcal{N}^{out}_i\to \mathcal{N}_{X_{t}(x)}$$   
satisfies $$H_i(B)\subset Dom(H^{r1}_x) \textrm{ and } \vert H^t_x(B)\vert_c=\vert H_x^{r_1}\circ H_i(B)\vert_c\geq e^{\lambda' (\tau_i+r)}b.$$ In particular, one has $$\vert H_x^{t}(B)\vert_c\geq e^{\tau_i+r_1}b>2b.$$ 
In this way, since $T_{\delta}<\infty$, we proceed in an inductively way to find a maximal $\ell\in\mathbb{N}$, positive numbers $r_1, r_2, \ldots, r_{\ell}$ such that $\tau+r_1+\cdots+r_{\ell}\leq t^1_x$, where $t^1_x>0$ is the first return time of the orbit of $x$ to $\mathcal{N}^{in}$,  and a connected component $B_{\ell}$ of $H_x^{\tau(x)+r_1+\ldots+r_{\ell}}(B)$ whose size is bigger than $2b$. 

By the choice of $\ell$, there is $s\geq 0$ such that $\tau+r_1+\cdots+r_{\ell}+s=t^1_x$. Besides, if $\hat{w}=\pi_{N_{X_{t_x^1-s}(x)}}(\Phi_{t^1_x-s}(\hat{v}))$, we have by \eqref{explane} and (a) that 
\begin{displaymath}
    \Vert \pi_{N_{X_{t_x^1}(x)}}(\Phi_{s}(\hat{w}))\Vert\geq \frac{\Vert X(x)\Vert}{\Vert X(X_{t}(x))\Vert}\cdot e^{\lambda s}\geq e^{\lambda s}>1.  
\end{displaymath}
Hence,  $\vert H_x^{t_x^1}(B)\vert_c>2b$. In particular, since $Y_{t^1_x}(x)\in \mathcal{N}^{in}$, it follows that there is a connected component $C_1\subset\mathcal{N}^{in}$ of $H_x^{t_x^1}(B)$ containing $x$ whose central volume is bigger than $2b$.

By repeating the previous process in an inductively way, there is $N\in\mathbb{N}$, and a connected component $C_N\subset \mathcal{N}^{in}$ of $H_x^{t_x^N}(B)$ whose central volume is bigger than $Nb>\frac{3}{2}d_s^i$ for every $i=0,\ldots, k$. Furthermore, we can choose $N$ and $C_N$ in a such way that $t_x^N>T$ and $C_N\subset \mathcal{N}_i^{in}$. So, we have that $X_{[0,s]}(C_N)$, $s>0$, is a submanifold of dimension $\dim F$ which intersects  $W^s(\sigma_j)$ transversely for some $j\in\lbrace 0,1,\ldots, k\rbrace$. Moreover, since the central volume of $C_N$ is bigger than $d^s_Y$, one has by (b) that $W^s(\gamma_j)\pitchfork X_{[0,s_0]}(C_N)\neq\emptyset$ for some $s_0>0$. So, the claim follows by taking $T_b=t_x^N+s_0$.

 To finish the proof, recall that the positive orbit of $x$ crosses $\mathcal{N}^{in}$ infinitely many times and, therefore, there is a sequence of times $t_n\to \infty$ and a singularity $\sigma^i\in Sing(X|_{C(X)})$ such that $X_{t_n}(x)\to x_0$ and $x_0\in W^s_{loc}(\sigma^i)$. Without loss of generality, we can assume $x_0,X_{t_n}(x)\in \mathcal{N}^{out}_i$, for every $n>0 $. Fix $\varepsilon$, $T>0$ and fix $N>0$ such that $d(X_{T_n}(x), x_0)\leq\varepsilon$, for every $n\geq N$. Fix $s_{\varepsilon}>T$ satisfying $d(X_{-s_{\varepsilon}}(x_0),\sigma^i)\leq \varepsilon$.
 Let $x_n\in \mathcal{N}^in$ be such that  $x_n=X_{t_N+s}(x)$ and $s>T$. We we consider a box $B$ around $x_n$ with diameter $\varepsilon$ and volume $b$, by the claim and the invariance of stable sets, there is $z\in B\cap W^s(\gamma_X)$. Therefore, there exist $s_z>T$ and $x_z\in \gamma_X$ such that $d(X_{s_z}(z),x_z)\leq \varepsilon$.  
 Finally,  the set $$\{\sigma^i,X_{-s_{\varepsilon}}(x_0),X_{t_n}(x),z,x_z\}$$ is an $\varepsilon$-$T$-chain from $\sigma^i$ to $\gamma_X$. Since $\sigma_X$ and $\sigma^i$ are chain related, the proof is complete.   
\end{proof}

Finally, we present the next result about the non-trivial dynamical behavior of the chain-recurrence classes $C(\sigma_Y)$ sufficiently $C^1$-close to a chain recurrent class $C(\sigma_X)$, with $X\in\mathcal{R}$ and such that $C(\sigma_X)$ is not reduced to either a homoclinic loop or a saddle connection between singularities.

\begin{lemma}\label{acummulationpoint}
Let $X\in \mathcal{R}$ and let $C(\sigma_X)\subset M$ be a non-trivial sectional-hyperbolic chain-recurrent class for $X$ which is not reduced to either a homoclinic loop or a saddle connection between singularities. Then, there is a neighborhood $\mathcal{U}_2\subset \mathcal{U}_1$ of $X$ such that for every $Y\in \mathcal{U}_2$ one has $C(\sigma_Y)\not\subset W^s(Sing(Y))$.
\end{lemma}

\begin{proof}
By Lemma \ref{connection1}, Lemma \ref{connection2} and item (2) of Lemma \ref{lemm: genericsets} we have that $C(\sigma_X)$ is a homoclinic class. Take two different periodic orbits $\gamma_0, \gamma_1$ in $C(\sigma_X)$, and let $d_{\gamma}=d_H(\gamma_0,\gamma_1)>0$. By continuity of $Y\in \mathcal{U} _X\to C(\gamma_Y)$ there is a neighborhood $\mathcal{U}_2\subset\mathcal{U}_1$ of $X$ such that
 \begin{equation}\label{closeperiodic}
d_H(C(\sigma_X),C(\sigma_Y))<\frac{d_{\gamma}}{2},\quad\forall Y\in\mathcal{U}_2.     
 \end{equation}
 Hence, if $C(\sigma_Y)\subset W^s(Sing(Y))$, then either it is a homoclinic loop or it is contained into  finitely many saddle connections $\gamma_{\alpha,\rho}$, associated to singularities $\alpha,\rho\in C(\sigma_Y)$. In any case, by \eqref{closeperiodic} one has $d_{\gamma}=d_H(\gamma_0,\gamma_1)<\frac{d_{\gamma}}{2}$, which is impossible. 
\end{proof}

Next, we are ready to prove Theorem \ref{Teohtop}. 

\begin{proof}[Proof of Theorem \ref{Teohtop}]
Let $\mathcal{R}$ be given by Lemma \ref{lemm: genericsets}. We consider the following cases:

\vspace{0.1in}
\textbf{Case 1: }$\Lambda$ does not contain singularities.
\vspace{0.1in}

By Lemma \ref{nsyngularcase}, it is a uniformly hyperbolic homoclinic class. In particular, there exists a neighborhood $\mathcal{U}_H$ of $\Lambda$ and a neighborhood $U$ of $\Lambda$ such that the maximal invariant set $H_Y$ contained in $U$ for any $Y\in\mathcal{U}_H$ is also hyperbolic of saddle type.

Now, by item (3) of Lemma \ref{lemm: genericsets}, there is a neighborhood $\mathcal{U}_X\subset\mathcal{U}_H$ of $X$ such that the continuation $\Lambda_Y$ of $\Lambda$ is contained in $H_Y$, so it is also hyperbolic of saddle type. Then, by compactness of $H_Y$ and the stable manifold Theorem, there are $\delta,\varepsilon>0$ such that 
\begin{equation}\label{transverseinnn}
    W_{\varepsilon}^s(z)\pitchfork W_{\varepsilon}^u(y)\neq\emptyset,\quad\forall z,y\text{ such that }d(z,y)<\delta.
\end{equation}

Let $x\in \Lambda$. Since $\Lambda$ is a homoclinic class there is a sequence of periodic orbits $\lbrace\gamma_n\rbrace_{n\geq 1}\subset\Lambda$ and a sequence of points $p_n\in\gamma_n$, $n\ge 1$, such that $p_n\to x$ when $n\to\infty$. Thus, by $n$ large enough, we can find a neighborhood $\mathcal{U}\subset \mathcal{U}_X$ of $X$ such that for any $Y\in\mathcal{U}$, there is a periodic orbit $\gamma_Y$ of $Y$ and a point $z\in\Lambda_Y$ such that $d(z,y)<\delta$. So, by \eqref{transverseinnn}, we have that $z\sim\gamma_Y$. This shows that $\gamma_Y\subset\Lambda$ by definition of chain class.

\vspace{0.1in}
\textbf{Case 2: }$\Lambda=C(\sigma_X)$ for some $\sigma_X\in Sing(X)$.
\vspace{0.1in}

By Lemma \ref{connection1}, Lemma \ref{connection2} and item (2) of Lemma \ref{lemm: genericsets} we have that $C(\sigma_X)$ is a homoclinic class and there is a neighborhood $\mathcal{U}_1$ of $X$ such that the conclusion of Lemma \ref{connection1} holds for every $Y\in\mathcal{U}_1$.

On the other hand, by item (3) of Lemma \ref{lemm: genericsets} and by follow the proof of Lemma \ref{connection2}, we deduce that the continuations $Sing(X|_{C(\sigma_Y)})$ consist entirely of Lorenz-like singularities, for any vector field $Y\in\mathcal{U}_2$.  
\begin{enumerate}
\item [(a)] For every $Y\in \mathcal{U}$, the singularities of $Y$ are Lorenz-like and $$\Vert Y(z)\Vert\leq \Vert Y(w)\Vert, \quad\forall z\in B(Sing(Y|_{C(\sigma_Y)}),\delta), w\notin B(Sing(Y|_{C(\sigma_Y)}),\delta).$$    
\item [(b)] For every $Y\in \mathcal{U}$ and $i=0,1,\ldots,k$, the submanifolds $\mathcal{N}^{in}_i$ and $\mathcal{N}^{out}_i$ are, respectively, cross-sections for $Y$ at $$p^Y_{s,i}\in C(\sigma^i_Y)\cap W^{s}(\sigma^i_Y) \textrm{ and }  p^Y_{y,i}\in C(\sigma^i_Y)\cap W^{u}(\sigma^i_Y),$$ where $p^Y_{s,i},p^Y_{u,i}\in \gamma_Y$ and $\gamma_Y$ is the periodic orbit given by Lemma \ref{connection1}. 
\item [(c)] For every $Y\in \mathcal{U}$ and $i=0,1,\ldots,k$, the holonomy map $H^Y_i$, and the cross-sections $\mathcal{N}^{in}_i$ and $\mathcal{N}^{out}_i$, with respect to the flow of $Y$, satisfy: $$(\mathcal{N}^{in}_i\setminus W^s_{loc}(\sigma^i_Y)) \subset dom(H_{i}^{Y}) \textrm{ and } H_i^Y(Dom(H^Y_i))\subset\mathcal{N}^{out}_i.$$   
    \item [(d)] For every $Y\in\mathcal{U}$ and $i=0,1,\ldots,k$, there exist $y^{s}_{0,i}, y^{u}_{0,i}\in C(\sigma_Y)$  such that $y^{s}_{0,i}\in \mathcal{N}^{in}_i$, $y^{u}_{0,i}\in \mathcal{N}^{out}_i$ and  
    \begin{displaymath}
    d(Q_{Y,i}^{s,u},y_{0,i}^{s,u}),d(Q_{Y,i}^{s,u},\gamma_{Y})\approx d^i_{s,u}, 
    \end{displaymath}
    where $Q^{s,u}_{Y_i}$ is the continuation of $Q^{s,u}_{X,i}$. This shows, in particular, that
\begin{displaymath}
    d^{s,u}_{Y,i}=\sup\lbrace d(y,Q^{s,u}_{Y,i}) : y\in C(\sigma_Y)\rbrace\geq d(Q^{s,u}_{Y,i},y_{0,i}^{s,u})\approx d^i_{s,u}. 
\end{displaymath}
In this case, define $$d^{s,u}_Y=\min\left\lbrace d^{s,u}_{Y,i} : i=0,\ldots,k\right\rbrace.$$
\end{enumerate}
By Lemma \ref{acummulationpoint} there is a neighborhood $\mathcal{U}_2\subset\mathcal{U}_1$ the continuation $C(\sigma_Y)$, there is a point $x\in C(\sigma_Y)\setminus W^s(Sing(Y))$ for every $Y\in\mathcal{U}_2$, so that,  by shrinking $\mathcal{U}_2$ if it is necessary, we follow Case 1 and Case 2 of the proof of Lemma \ref{connection2} verbatim, replacing $X$ for $Y$, to get a periodic orbit $\gamma_Y$ such that $\sigma_Y\sim\gamma_Y$ and $\gamma_Y\sim \sigma_Y$. This shows that $\gamma_Y\subset C(\sigma_Y)$ by definition of chain class. 
\end{proof}

\section{Proof of the Main Theorem}

 This section is devoted to prove our main results. Before proving the Main Theorem, we briefly outline the main idea behind the argument. The strategy combines the density of periodic orbits in non-trivial sectional-hyperbolic chain-recurrence classes, and the denseness of its stable and unstable manifolds. 
In this way, for proving the Main Theorem  we need to prove the following lemma:

\begin{lemma}\label{lemamatador}
Let $\mathcal{R}$ be the generic set given by Lemma \ref{lemm: genericsets},  $X\in \mathcal{R}$ and $C(\sigma)$ be a non-trivial chain-recurrent class  of $X$ containing a periodic orbit $\gamma$. There is an open neighborhood $\mathcal{U}_X$ of $X$ such that if $Y\in \mathcal{U}_X$, then for every $x\in C(\sigma_Y)$ and every neighborhood $V_x$ of $x$, we have    $$ W^s(\gamma_Y)\cap V_x\neq\emptyset \textrm{ and }  W^u(\gamma_Y)\cap V_x\neq\emptyset,$$
where $\gamma_Y\subset C(\sigma_Y)$ is a periodic orbit.
\end{lemma}

To improving our exposition, we shall split the proof of Lemma \ref{lemamatador} into four separate lemmas. Let us begin with the first of them which is stated as follows:

 \begin{lemma}[Theorem 3.1 in \cite{CY}]\label{stablem}
     There is a residual set $\mathcal{R}\subset \mathcal{X}^1(M)$ such that if $X\in \mathcal{R}$ and $C(\sigma)$ is a non-trivial sectional-hyperbolic chain-recurrent class of $X$ and $\gamma\subset C(\sigma)$, then there are neighborhood $\mathcal{U}$ of $X$ and a neighborhood $U$ of $C(\sigma)$ such that:

     If $Y\subset \mathcal{U}$ and $x\in C(\sigma_Y)$, then either $x$ is on the unstable manifold of a singularity or $W^{ss}_{loc,Y}(x)\pitchfork W_Y^{u}(\gamma_Y)\neq\emptyset$. 
 \end{lemma}
\begin{remark}
    The attentive reader will notice that Lemma~\ref{stablem} is stated under slightly different assumptions than in \cite{CY}, where it is assumed that $C(\sigma)$ is a Lyapunov-stable class. However, a careful analysis of the proof in~\cite{CY} reveals that this additional assumption is not actually used. Therefore, the same argument applies verbatim to the version we present here. For this reason, we omit the proof of Lemma~\ref{stablem} to avoid unnecessary repetition.
\end{remark}

The next lemma gives us one of the inequalities in Lemma \ref{lemamatador}:

\begin{lemma}\label{lem: stable}
    Under the assumption of Theorem \ref{Teohtop}, $W^u(\gamma_Y)$ is dense on $C(\sigma_Y)$.
\end{lemma}
\begin{proof}
  Let $Y\in \mathcal{U}$, $x\in C(\sigma_Y)$ and let $V$ be a neighborhood of $x$. We claim that $W^u(\gamma_Y)\cap V\neq \emptyset$. Indeed, we have the following cases: 

\textbf{Case 1:} Assume $x\in W^u(\rho)$, for some $\rho\in C(\sigma_Y)$. In this case, for $\varepsilon>0$ there is $t>0$ such that $d(\sigma,Y_{-t}(x))<\varepsilon$. Take a neighborhood $W\subset M$ of $Y_{-t}(x)$ such that $Y_t(W)\subset V$. By Remark \ref{remark3.5} one has $W^u(\gamma_Y)\pitchfork W^s(\rho)$. So, if $\varepsilon$ is small enough, we have by the inclination lemma that there is a point $y_t\in W^u(\gamma_Y)\cap W$, so that $Y_t(y_t)\in V\cap W^u(\gamma_Y)$.  

\textbf{Case 2:} Assume $x\notin W^u(\sigma'_Y)$ and let $\varepsilon>0$ such that $B(x,\varepsilon)\subset V$. By Lemma \ref{stablem}, for every $t>0$, one has there is $y_{-t}\in W^{ss}_{loc}(Y_{-t}(x))\pitchfork W^u(\gamma_Y)$. Since $W^{ss}_{loc}(Y_t(x))$ is uniformly contracted by sectional-hyperbolicity, with rate $\lambda_Y>0$, we take $t>0$ in a such way that $d(Y_t(y_t),x)\leq Ce^{-\lambda_Yt}<\varepsilon$. So, $Y_t(W^{ss}_{loc}(Y_{-t}(x))\subset V$. This shows that $Y_t(y_t)\in V\cap W^u(\gamma_Y)$, and therefore $W^u(\gamma_Y)\cap V\neq \emptyset$. 
\end{proof}

For the next lemmas we need to recall another type of cross-section that were defined in \cite{CY}. These cross-section are different from the cross-sections introduced in Section \ref{section:robustper} and hence here we call them adapted cross-sections to avoid confusion.  Recall that since $X$ is a generic vector field, one can assume that $X$ has finitely many hyperbolic singularities. Thus, we can suppose the same follows for the any vector field in $\mathcal{U}$.  Following \cite[Section 4.2]{CY} for any  singularity singularity $\sigma\in C(\sigma_Y)$, there is a co-dimension one  submanifold $S_\sigma$ so that for every $\sigma\in C(\sigma_Y)$ it 
    holds:
    \begin{itemize}
        \item The submanifolds $\overline{S_\sigma}$ are pairwise disjoint
        \item $\overline{S_\sigma}\cap Sing(Y)=\emptyset$ and the angle between $Y(x)$ and $T_xS_\sigma$ is bigger than $\alpha_\sigma>0$, for every $x\in S_\sigma$.
    \end{itemize}
    
    By considering $\alpha=\min\{\alpha_\sigma; \sigma\in Sing(Y)\}$ and $S_Y=\cup_{\sigma \in Sing(Y)} S_\sigma$ it holds that there is a compact set $\Delta\subset S_Y$ such that 

    \begin{itemize}
        \item $\overline{S_Y}\cap Sing(Y)=\emptyset$ and the angle between $Y(x)$ and $T_xS_Y$ is bigger than $\alpha>0$, for every $x\in S_Y$.
        \item The interior of $\Delta_Y$ intersects the any forward orbit in $M\setminus \cup_{\sigma\in Sing(Y)} W^s_{loc}(\sigma)$. 
        \item For every $\sigma\in Sing(Y)$, $S_Y\cap W^s_{loc}(\sigma)$ and $S_Y\cap W^u(\sigma)$ are a one-codimensional spheres inside $W^s_{loc}(\sigma)$ and  $W^u_{loc}(\sigma)$, respectively.
    \end{itemize}
 Hereafter, the sets $S_\sigma$ and $S_Y$ will be called the \textit{adapted cross-sections of $Y$}.

Now, since $M$ is a compact manifold, there is a well-defined injectivity radius $e_0>0$ for the exponential map $\exp_x$, for every $x\in M$. For every regular point $x\in C(\sigma_Y)$ and $\delta>0$, we say that $D_{\delta}(x)\subset M$ is a {\it center-unstable  section disc} (or simply {\it cu-section disc}) of $Y$, if it is given by $\exp_x(B^{c-1}_{\delta}(0))$, where $B^{c-1}_{\delta}(0)\subset T_xM$ is the $c-1$-dimensional ball centered at $x$ and radius $\delta$, and $T_yD_{\delta}(x)\oplus \langle Y(y)\rangle\subset \C_y^c(a)$ for every $y\in D_{\delta}(x)$.  Let $e_0>0$ be the injectivity radius of $\exp$. Fix $0<\varepsilon<e_0$ so that $B_{\varepsilon}(C(\sigma_Y))\subset U$, where $Y\in\mathcal{U}$ and $U\subset M$ are given by Theorem \ref{Teohtop}. Given $T>0$, we can shrinking $U$ in a such way that $Y_{[0,T]}(x)\in U$. The next lemma shows that these cu-section discs expands, in a certain way, along the orbit of points of $C(\sigma_Y)$.

\begin{lemma}\label{lem:bigdiscs}
    Suppose we are under the assumptions    of Theorem \ref{Teohtop}. Let $Y\in \mathcal{U}$ and fix  $\varepsilon>0$ so that $B_{\varepsilon}(C(\sigma_Y))\subset U$. Then, for every regular point $x\in C(\sigma_Y)$, $\delta>0$, and a $cu$-section disc $D_{\delta}(x)$, with $diam(D_{\delta}(x))\leq\delta$, containing $x$, there is $T'=T'(x,\delta)>0$ such that $Y_{t}(D_x)$ contains a $cu$-disc section $D_{\varepsilon}(Y_t(x))$, for every $t\geq T'$.
\end{lemma}

\begin{proof}
Let $e_0>0$ be the injectivity radius of $\exp$. Fix $0<\varepsilon<e_0$ so that $B_{\varepsilon}(C(\sigma_Y))\subset U$, and let consider $0<\delta<\varepsilon$. Let $S_Y$ be an adapted cross-section for $Y$. Thus, for every regular point $x\in C(\sigma_Y)$ there are $\sigma \in Sing(X)$ $S_{\sigma}$ and $t_x'>0$ such that $Y_{t_x'}(x)\in S_\sigma=S_x$. In particular, there are $\alpha_1,\alpha_2,\beta>0$ such that $\alpha_1\leq|Y(Y_{t_x'}(x))|\leq \alpha_2$, and the angle between $Y(Y_{t_x'}(x))$ and $S_x$ is bigger than $\beta$. Therefore, 
we have $DY_{t}(\C^c(y))\subset \C^c(Y_{t}(y))$, for every $y\in D_{\delta}(x)$ and every $t>0$ satisfying $Y_t(y)\subset U$. Hence, $$D_{r_t}(Y_t(x))=Y_t(D_{\delta}(x))\cap \exp_{Y_t(x)}B_{\varepsilon}(0),$$ where $t>0$ and $0<r_t\leq \varepsilon$, is a $cu$-section disc of $Y$. In other words, the $cu$-section discs are invariant for the flow. 

Now, without loss of generality we assume that $t_x'=0$, $D=D_{\delta}(x)\subset S_x$ for some cross section $S_x$ for $Y$. Note that, by definition, $D$ is a connected set. Then, if there is $y\in D$ and $t>0$ such that $Y_t(y)\notin B_{\varepsilon}(C(\sigma_Y))$, there is a connected component $D'\subset D$ such that $diam\; Y_t(D')\geq \varepsilon$, so the result is obtained in this case. 

Assume that $Y_t(D)\subset B_{\varepsilon}(C(\sigma_Y))$ for every $t>0$. Let $y\in D$, $v\in T_{y}D$ be a unit vector, and denote by $L_y$  the parallelogram formed by $\langle Y(y)\rangle$ and $v$. Take $t_0>0$ such that $C''\sin(\beta)e^{\lambda t_0}\geq\frac{\varepsilon}{\delta}$, where $C''=C'\frac{\alpha_1}{\alpha_2}\sin(\beta)>0$ and $C'=CK$, where $C,K>0$ are uniform constants given by sectional-hyperbolicity and the choice of the cross sections, and $Y_{t_0}(x)$ belongs to some cross section $S'$ for $Y$. Therefore, we have that 
\begin{eqnarray*}
    \Vert \Phi_{t_0}(v)\Vert&=&K\cdot\frac{Area(\Phi^Y_{t_0}(L_y))}{\Vert Y(Y_{t_0}(x))\Vert}\\
    &=&K\cdot\frac{\det(\Phi^Y_{t_0}(x))\cdot Area(L_y)}{\Vert Y(Y_{t_0}(y))\Vert}\\
    &\geq& KCe^{\lambda t_0} \cdot \frac{|| Y(y)||\cdot\sin(\theta)}{\Vert Y(Y_{t_0}(y))\Vert}\\
    &\geq&C'\frac{\alpha_1}{\alpha_2}\sin(\beta)e^{\lambda t_0}\\
    &=& C''e^{\lambda t_0}\geq \frac{\varepsilon}{\delta},
\end{eqnarray*}
where  $\theta$ is the angle between $Y(y)$ and $v$. So,  
 $diam\;Y_{t}(D)\geq \frac{\varepsilon}{\delta}\cdot diam\;D=\varepsilon$, for every $t\geq t_0$. 
This concludes the proof.
\end{proof}

\begin{lemma}\label{densestable}
    Under the assumption of Theorem \ref{Teohtop}, $W^s(\gamma_Y)$ is dense on $C(\sigma_Y)$.
\end{lemma}
\begin{proof}
Let $Y\in \mathcal{U}$, $x\in C(\sigma_Y)$ and let $V$ be a neighborhood of $x$. We will prove that $W^{s}(\gamma_Y)\cap V\neq \emptyset$. First, we assume that $x$ is regular. In this case, fix  $0<\varepsilon<e_0$ such that $B_{\varepsilon}(C(\sigma_Y)\subset U$, let consider a $cu$-section disc $D=D_{\delta}(x)$, $0<\delta<\varepsilon$, such that $D\subset V$, and let $T'=T'(x,\delta)>0$ be given by Lemma \ref{lem:bigdiscs}. According to the choice of the neighborhood $U$ and the invariance of the $cu$-discs, there is $\delta_0>0$ such that the disc $D_{\delta_0}=\bigcup_{\vert t\vert<\delta_0}Y_t(D)$, of dimension $dim\;F$, is tangent to $\C^c$ and it is contained in $V$. Moreover, by following the proof of Lemma \ref{lem:bigdiscs}, one has that the inner radius of $Y_t(D_{\delta_0})$ increases with $t>0$.   

Second, recall that the angle between $E$ and $F$ is uniformly bounded away from zero. So, up to possibly reducing $\varepsilon$, there is $\eta>0$ satisfying the following property: If $y,z\in C(\sigma_Y)$ satisfy  
$d(y,z)\leq \eta$ and $D(y)$, $D(z)$ are discs tangent to the center unstable cone $\C^c$ with inner radius $\varepsilon$ and $0<r_z<\varepsilon/2$ respectively, then
\begin{equation}\label{intersectiondiscs}
W^s_{loc}(q)\pitchfork D(y)\neq\emptyset,\quad \forall  q\in D(z).   
\end{equation}   

Now, let $t>0$ such that the inner radius of $Y_t(D_{\delta_0})$ contains a $cu$-section disc at $x_t=Y_t(x)$ with inner radius $\varepsilon$. Since $C(\sigma_Y)$ is a chain-recurrent class, it is also chain-transitive. Consequently, for every $\eta>0$ there is a $\eta$-chain $(x_i,t_i)_{i=0}^k$, with $t_i\geq 1$, such that $x_0=x_t$, $x_k=p\in\gamma_Y$, $d(Y_{t_i}(x_i),x_{i+1})\leq \eta$, and $x_i\in C(\sigma_Y)$ for every $i=0,..,k$. Then, if we take a $cu$-section disc $D_{\varepsilon}(Y_{t_{k-1}}(x_{k-1}))$,  by the choice of $\eta$ and \eqref{intersectiondiscs} we obtain a point $z_{k}$ in $W^s_{loc}(\gamma_Y)\pitchfork D_{\varepsilon}(Y_{t_{k-1}}(x_{k-1}))\neq\emptyset$. Moreover, since the inner radius of $Y_{-t_i}(D_{\varepsilon}(Y_{t_i}(x_i)))$ is less than $\varepsilon/2$ for every $i=1,\ldots, k-1$, we obtain in a recursive way, points $z_{i}$ in $W^s_{loc}(\gamma_Y)\pitchfork D_{\varepsilon}(Y_{t_i}(x_i))$. The Figure 3 illustrates how these points are obtained. 
\begin{figure}[ht]
     \centering
 \includegraphics[scale=0.4]{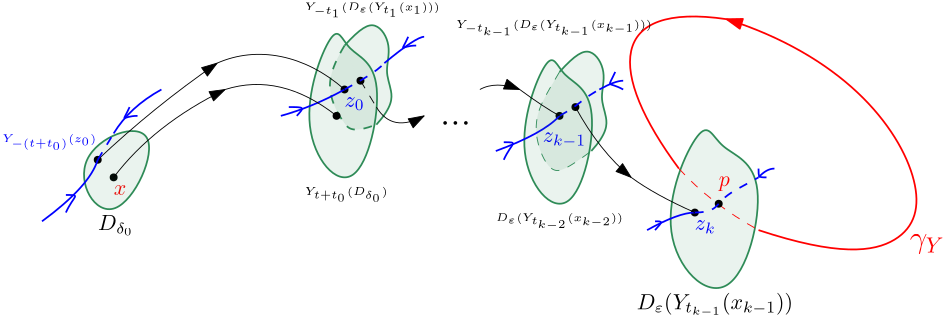}
   \caption{Proof of Lemma \ref{densestable}.}
    \label{fig:my_label}\label{fig: robustperiodic}
 \end{figure}

In this way, since the inner radius of $Y_{t+t_0}(D_{\delta_0})$ is at least $\varepsilon$, we have by definition of $\eta$ a point $z_0\in W^s_{loc}(\gamma_Y)\pitchfork Y_{t+t_0}(D_{\delta_0})$. So, the point $Y_{-(t+t_0)}(z_0)\in W^s(\gamma_Y)\cap D_{\delta_0}\subset V$. 

Finally, if $\rho$ is a singularity of $C(\sigma_Y)$ and $V\subset M$ is an open set containing $\rho$, there is a regular point $x\in C(\sigma_Y)$ and an open neighborhoo $V_x$ of $x$ satisfying $V_X\subset V$. So, by the above argument, there is a point $z\in W^s(\gamma_Y)\cap V_x\subset V$. This concludes the proof. 
\end{proof}

Next, we are ready to prove the Main Theorem.

\begin{proof}[Proof of the Main Theorem]
Consider the residual set $\mathcal{R}$ given by Lemma \ref{lemm: genericsets}. Let $X\in \mathcal{R}$ and let $\Lambda$ be a non-trivial sectional-hyperbolic chain-recurrent of $X$ which is not reduced to either a homoclinic loop or a saddle connection between singularities. 

By Theorem \ref{Teohtop} we have that $\Lambda$ is a homoclinic class. Consequently, there is a periodic orbit $\gamma\subset \Lambda$ whose stable manifold is dense in $\Lambda$ and meets transversally every submanifold of dimension $\dim(F)$ that is tangent to $\C^c$, has inner radius $\varepsilon$, and intersects a sufficiently small neighborhood $U' \subset U$ of $\Lambda$. Then, following the arguments in the proof of \cite[Theorem 4.1]{CY}, we can see that this property holds for vector fields $Y$ in a small $C^1$-neighborhood $\mathcal{U}_X$ of $X$. By Theorem \ref{Teohtop} we also have that $\Lambda_Y$ contains periodic orbits. Moreover, for every $\gamma'\subset \Lambda_Y$ and a neighborhood $U_{\gamma'}\subset U'$ of $\gamma'$, one has that $W_{\varepsilon}^u(\gamma')$ is a submanifold tangent to $\C^c$ with inner radius $\varepsilon$. Thus, by Lemma \ref{lemamatador} 
we have $W^s(\gamma_Y)\cap W_{\varepsilon}^u(\gamma')\neq\emptyset$. In particular, this shows that the set of transverse intersections $W^s(\gamma_Y)\pitchfork W^u(\gamma')$ is dense in $ W^u(\gamma')$.

Therefore, if $x\in \Lambda_Y$ and $V_x\subset U$ is a neighborhood of $x$, by Lemma \ref{lemamatador} there is a point $y\in W^u(\gamma')$ for some periodic orbit $\gamma'\subset \Lambda_Y$, which in turns is accumulated by elements of $W^s(\gamma_Y)\pitchfork W^u(\gamma')$. In other words, $\Lambda_Y$ is robustly  a homoclinic class. This proves the result.  
\end{proof}




\textbf{Acknowledgments.} The authors would like to thank some colleagues for valuable comments that helped them to improve the presentation of this work. Precisely, we thank Professor Alexander Arbieto and Professor Fan Yang for the careful reading of this manuscript and for pointing out several issues in an earlier version of this work. K. J. Vivas was supported by ANID Proyecto FONDECYT Iniciaci\'on 11250633, Chile.
\begin{table}[h]
\begin{tabularx}{\linewidth}{p{1.5cm}  X}
\includegraphics [width=1.8cm]{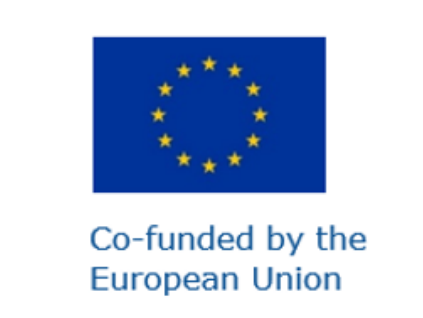} &
\vspace{-1.5cm}
This research is part of a project that has received funding from
the European Union's European Research Council Marie Sklodowska-Curie Project No. 101151716 -- TMSHADS -- HORIZON--MSCA--2023--PF--01.\\
\end{tabularx}
\end{table}

\address


\begin{thebibliography}{aaaaa}

\bibitem[AN07]{AN} J.~M. Alongi and G.~S. Nelson.
\newblock {\it Recurrence and topology},
\newblock Graduate Studies in Mathematics, 85, Amer. Math. Soc., Providence, RI, 2007; 


\bibitem[ABS77]{ABS} 
Afraimovich V. S., Bykov V. V. and Shilnikov L. P.
\newblock{ On the appearance and structure of
the Lorenz attractor}, 
\newblock{\em Dokl. Acad. Sci. USSR.} 234: 336--339, 1977.

\bibitem[AM17]{AM} V. Araujo and  I. Melbourne.
\newblock Existence and smoothness of the stable foliation for sectional-hyperbolic attractors.
\newblock{\em Bull. London Math. Soc.} 49: 351--367, 2017.

\bibitem[AM19]{AM2} V. Araujo and  I. Melbourne.
\newblock Mixing properties and statistical limit theorems for singular hyperbolic flows without a smooth stable foliation. \newblock{Adv. in Math.} 349: 212--245, 2019.

\bibitem[Ba04]{B} S. Bautista. 
\newblock{ The geometric Lorenz attractor is a homoclinic class.}
\newblock{\em Bol. Mat. (N.S.)} 11: 69--78, 2004.

\bibitem[BM11]{BM} S. Bautista and  C. Morales. 
\newblock Lectures on sectional-Anosov flows. 
\newblock{\em Preprint IMPA S\'erie D.} 84, 2011.


\bibitem[BC04]{BC} C. Bonatti and S. Crovisier.
\newblock R\'ecurrence et g\'en\'ericit\'e. 
\newblock{\em Invent. Math.}, 158: 33--104, 2004.


\bibitem[BdL21]{BdL21}
C. Bonatti; A. da Luz.
\newblock { Star flows and multisingular hyperbolicity}, 
\newblock {\em Journal of the European Mathematical Society}  23, no. 8, pp. 2649--2705 2021

\bibitem[BPV97]{BPV}
C. Bonatti; A. Pumari\~no and M. Viana.
\newblock { Lorenz attractors with arbitrary expanding dimension}, 
\newblock {\em C. R. Acad. Sci. Paris S\'er. I Math.} 325: 883--888, 1997.



\bibitem[Con78]{Co} C. Conley
\newblock  Isolated invariant sets and Morse index.
\newblock{\em CBMS Regional Conference Series in
 Mathematics}, 38, AMS (1978)

\bibitem[Cro06]{C} S. Crovisier.
\newblock Periodic orbits and chain-transitive sets of $C^1$-diffeomorphisms. 
\newblock{Publ. Math.
Inst. Hautes \'Etudes Sci.}, 104: 87--141, 2006.

\bibitem[Cro10]{Cro} S. Crovisier.
\newblock Birth of homoclinic intersections: a model for the central dynamics of partially
 hyperbolic systems. 
\newblock{\em Ann. of Math.}, 172: 1641--1677, 2010.

\bibitem[CWYZ24]{CWYZ}  S. Crovisier, X. Wang, D. Yang and J. Zhang. 
\newblock On physical measures for multi-singular hyperbolic vector fields.
\newblock{\em Trans. Amer. Math. Soc. }, 377: 10 (2024), 6937--6980, 2024.


\bibitem[CY21]{CY}  S. Crovisier and D. Yang. 
\newblock Robust transitivity of singular hyperbolic attractors. 
\newblock{\em Mathematische Zeitschrift}, 298: 469--488, 2021.



\bibitem[FS76]{FS} J.E. Franke and J. F. Selgrade, 
\newblock Abstract $\omega $-limit sets, chain-recurrent sets, and basic sets for flows.
\newblock {\em Proc. Amer. Math. Soc.},  60 (1976), 309--316 (1977)




\bibitem[GW06]{GW06} Gan S. and Wen L. 
\newblock   Nonsingular star flows satisfy Axiom A and the no-cycle condition,
\newblock{\em Invent. Math.} 164, 279--315, 2006. 



\bibitem[GY18]{GY} Gan S. and Yang D. 
\newblock  Morse-Smale systems and horseshoes for three dimensional singular flows 
\newblock{\em Annales Scienti ques de I \'Ecole Normale Sup\'erieure} 51, 39--112 (2018). 

\bibitem[GYZ22]{GYZ} Gan S., Yang J. and Zheng R.
\newblock Lyapunov stable chain-recurrence classes for singular flows. 
\newblock{Pre-Print} arXiv:2202.09742.

\bibitem[Gu76]{G} Guckenheimer J.
\newblock A strange, strange attractor. 
\newblock{The Hopf bifurcation and its application, Applied Mathematical Sciences} 19, 1976.



\bibitem[Ha92]{Ha92} Hayashi S.
\newblock  Diffeomorphisms in $\mathcal{F}^1
(M)$ satisfy Axiom A 
\newblock{\em Ergodic Theory \& Dynam. Systems} 12, 233--253, 1992. 


\bibitem[HPPS70]{HPPS} M. Hirsch, J. Palis, C. Pugh, M. Shub. 
\newblock Neighborhoods of hyperbolic sets. 
\newblock{\em Invent. Math.} 9: 121--134, 1970.

\bibitem[Li81]{Li81} Liao S.
\newblock  Obstruction sets. II. (in Chinese)  
\newblock{\em Beijing Daxue Xuebao} 2, 1-36, 1981. 

\bibitem[dL17]{dL17} da Luz A.
\newblock  Hyperbolic sets that are not contained in a locally maximal one.  
\newblock{\em Discrete Contin.
Dyn. Syst.} 37, no. 9, 4923--4941, 2017. 


\bibitem[Ma78]{Ma78} Ma\~n\'e R.
\newblock  Contributions to the stability conjecture.    
\newblock{\em Topology} 17, no. 4, 383-396, 1978. 



\bibitem[MM08]{Me3}  R. Metzger and C. Morales. 
\newblock{ Sectional-Hyperbolic systems.}
\newblock{\em Ergodic Theory Dynam. Systems.} 28:1587--1597, 2008.

\bibitem[MPP99]{Mo3} C. Morales, M. J. Pacifico and E. Pujals.
\newblock{ Singular hyperbolic systems.}
\newblock{\em Proc. Amer. Math. Soc.} 127: 3393--3401, 1999.

\bibitem[New72]{N} S. Newhouse.
\newblock{ Hyperbolic Limit Sets.}
\newblock{\em Trans. Amer. Math. Soc.} 167: 125--150, 1972.

\bibitem[PYY21]{PYY} M. J. Pacifico, F. Yang and J. Yang. 
\newblock Entropy theory for sectional hyperbolic flows.
\newblock{\em Annales de L'Institut Henri Poincar\'e C, Analyse non line\'aire.} 38: 1001--1030, 2021.

\bibitem[PYY23]{PYY2} M. J. Pacifico, F. Yang and J. Yang.
\newblock An entropy dichotomy for singular star flows.
\newblock{\em Trans. Amer. Math. Soc}. 376: 6845--6871, 2023.

\bibitem[P72]{P} V. Pliss.
 \newblock  A hypothesis due to Smale.
 \newblock{\em  Diff. Eq.}  8 (1972), 203--214, 1972. 


\bibitem[Sa19]{Sal} L. Salgado.
\newblock Singular hyperbolicity and sectional Lyapunov exponents of various orders.
\newblock{\em Proc. Amer. Math. Soc.}  147: no. 2, 735--749, 2019. 



\bibitem[SGW14]{SGW14} Y. Shi, S. Gan and L. Wen.
\newblock  On the singular hyperbolicity of star flows.
\newblock{\em Journal of  Modern Dynamics}  8, 191--219, 2014.

\bibitem[Sm67]{S}
S. Smale.
\newblock Differentiable dynamical systems.
\newblock {\em Bull. Amer. Math. Soc.} 73: 747--817, 1967.



\bibitem[XZ08]{XZ} Xiao Q. and Zheng Z. 
\newblock $C^1$ weak Palis conjecture for nonsingular flows.
\newblock{\em  Discrete \& Continuous Dynamical Systems} 38, 1809--1832 (2008). 


\bibitem[ZGW08]{ZGW08}  Zhu S., Gan S. and  Wen. L. 
\newblock Indices of singularities of robustly transitive sets.
\newblock{\em  Discrete \& Continuous Dynamical Systems} 21 , 945--957 (2008). 

\end{thebibliography}
\end{document}